\documentclass[11pt,oneside,reqno]{amsart}

\usepackage{amsmath,amssymb,amsthm,textcomp,mathtools}
\usepackage{amsfonts,graphicx}
\usepackage[mathscr]{eucal}
\usepackage{color}
\usepackage{csquotes}
\usepackage{enumitem}
\numberwithin{equation}{section}

\theoremstyle{definition}

\numberwithin{equation}{section}

\newtheorem{theorem}{\bf Theorem}[section]
\newtheorem{remark}{\bf Remark}[section]
\newtheorem{proposition}{Proposition}[section]

\newtheorem{definition}{Definition}[section]
\newtheoremstyle
{remarkstyle}
{}
{11pt}
{}
{}
{\bfseries}
{:}
{     }
{\thmname{#1} \thmnumber{#2} }

\theoremstyle{remarkstyle}

\begin{document}
	\title{Skellam Random Fields and Their Fractional Variants}
	 \author[Pradeep Vishwakarma]{Pradeep Vishwakarma}
	 \address{Pradeep Vishwakarma, Theoretical Statistics and Mathematics Unit,
	 	Indian Statistical Institute, Kolkata, 700108, West Bengal, India.}
	 \email{vishwakarmapr.rs@gmail.com}	
	\date{\today}	
	\maketitle
	\begin{abstract}
		We study some Skellam-type spatial point processes. As a particular case, we consider a Skellam random field (SRF) on the positive quadrant of the plane, which is a two parameter L\'evy process with rectangular increments. A weak convergence result is obtained for the SRF. The Riemann-Liouville integral of the SRF over finite rectangles is analyzed. We derive a scaled compound Poisson field characterization for the Riemann integral of the SRF. Also, an explicit expression of its characteristic function is obtained.  Later, we consider three fractional variants of the two parameter SRF. Their point probabilities, associated governing equations, and various other distributional properties are analyzed. 
	\end{abstract}
	
\section{Introduction}\label{sec1}
 The Skellam-type distribution was introduced and studied in \cite{Irwin1937, Skellam1946}. The standard Skellam process is an integer-valued point process, defined as a difference of two independent Poisson processes. Let $\{N_1(t),\ t\ge0\}$ and $\{N_2(t),\ t\ge0\}$ be independent Poisson processes with positive transition rates $\lambda_1$ and $\lambda_2$, respectively. The one parameter Skellam process $\{S(t),\ t\ge0\}$ is defined as follows (see \cite{Barndorff-Nielsen2011}):
 \begin{equation}\label{spdef}
     S(t)\coloneqq N_1(t)-N_2(t),\ t\ge0.
 \end{equation}
 Its probability generating function is given by
 \begin{equation}\label{sppgf}
 	G(u,t)=\exp\bigg(\lambda_1t(u-1)+\lambda_2t\bigg(\frac{1}{u}-1\bigg)\bigg),\ 0<|u|\leq1.
 \end{equation} The Skellam process has been widely studied for both its theoretical properties and practical applications. From insurance to image processing, it finds applications in a variety of fields. For example, Barndorff-Nielsen et al. \cite{Barndorff-Nielsen2011} introduced a scaled and generalized version of the Skellam process using negative binomial distributions for financial modeling. In \cite{Kerss2014}, a time-changed Skellam process is analyzed, and its potential application in finance is discussed. For a more detailed study on the difference between two Poisson processes, we refer the reader to \cite{Barndorff-Nielsen2011, Carr2011}.
The fractional variants of the Skellam process are introduced in \cite{Kerss2014}. Skellam processes of order $k$ and their time-changed variants can be found in \cite{Gupta2020, Kataria2024a}, and references therein. Moreover, a Skellam process via a generalized counting process and its time-changed variants are introduced and studied in \cite{Kataria2022}. Recently, a generalisation of the Skellam point process and its fractional variants are investigated in \cite{Cinque2025}.  

 The spatial point process on finite dimensional Euclidean space describes the random distribution of points within finite rectangles. The spatial Poisson point process (Poisson random field) is one of the most famous point processes for modeling the random point patterns in spatial locations. It provides a key statistical model for analyzing point patterns in higher dimensions, where each point corresponds to the spatial position of an object or an event. Such processes find applications in various fields of engineering, physics, astronomy, ecology, telecommunication, \textit{etc}. In recent years, various fractional variants of the Poisson point process on the plane have been investigated by many authors, for example, see \cite{Aletti2018, Kataria2025, Leonenko2015, Vishwakarma2025}. These models provide advanced mathematical tools for modeling the complex spatial data exhibiting a long correlation structure.

The spatial extension of the standard Skellam process on the real line leads to an integer-valued point process. We note that such point processes have not yet been explored in the existing literature. In this paper, we investigate and characterize several classes of Skellam random fields arising within the broader family of Poisson random fields. We introduce a Skellam-type point process defined via a Poisson point process on a finite dimensional Euclidean space. Let $\mathbb{R}_+$ be the set of non-negative real numbers and $\mathcal{B}$ be the set of Borel sets of $\mathbb{R}^M_+$, $M\ge1$. Let $\mathcal{J}\subset\mathbb{R}-\{0\}$ be a finite subset. We consider a random field  defined as 
\begin{equation*}
    \mathcal{S}(B)\coloneqq\sum_{j\in\mathcal{J}}jN_j(B),\ B\in\mathcal{B},
\end{equation*}
where $\{N_j(B),\ B\in\mathcal{B}\}$, $j\in\mathcal{J}$ are independent Poisson random fields on $\mathbb{R}^M_+$ (for definition, see Section \ref{prfdef}). Note that $\{\mathcal{S}(B),\ B\in\mathcal{B}\}$ is an integer-valued random measure, and we call it the generalized Skellam random field (GSRF). The GSRF is the spatial version of a generalized Skellam point process introduced and studied in \cite{Cinque2025}. In particular, for $M=1$, it reduces to the same. For $\mathcal{J}=\{-1,1\}$, the GSRF reduces to the spatial version of the standard Skellam process as defined in (\ref{spdef}). We denote it by $\{S(B),\ B\in\mathcal{B}\}$ and called it the Skellam random field (SRF). It is defined as
\begin{equation*}
    S(B)\coloneqq N_1(B)-N_2(B),\ B\in\mathcal{B},
\end{equation*}
where $\{N_1(B),\ B\in\mathcal{B}\}$ and $\{N_2(B),\ B\in\mathcal{B}\}$ are independent Poisson random fields with rate parameters $\lambda_1>0$ and $\lambda_2>0$, respectively.
The one dimensional distribution of SRF is given by
\begin{equation*}
    \mathrm{Pr}\{S(B)=n\}=e^{-(\lambda_1+\lambda_2)|B|}\bigg(\frac{\lambda_1}{\lambda_2}\bigg)^{n/2}I_{|n|}(2\sqrt{\lambda_1\lambda_2}|B|),\ n\in\mathbb{Z},
\end{equation*}
where $|B|$ denotes the Lebesgue measure of $B$, and $\mathbb{Z}$ is the set of integers. Here, $I_\nu$, $\nu>-1$ is the modified Bessel function of the first kind defined as
\begin{equation}\label{Besselm}
	I_\nu(x)=\sum_{k=0}^{\infty}\frac{(x/2)^{2k+\nu}}{k!\Gamma(\nu+k+1)},\ x\in\mathbb{R}.
\end{equation}

The Skellam random field provides an integer-valued point process for modeling the spatial random point patterns in which the positive value represents the addition of points or net gain, and the negative value represents the removal of points or net loss. It can be used to model real-world systems across various applied fields. For example, in image analysis, it can be used to model the appearance of points on the detector screen; in wireless networking, it is used to model the status of node locations; and in epidemiology, it is used to study the outbreak of a disease. Moreover, fractional variants of the SRF provide mathematical tools for modeling the aforementioned systems whenever they exhibit complex temporal behavior.

A brief outline of the paper is as follows:

In Section \ref{sec2}, we collect some known definitions and results about the Poisson random field and its time-fractional variant. In Section \ref{sec3}, we begin by introducing the GSRF $\{\mathcal{S}(B),\ B\in\mathcal{B}\}$. We obtain a compound Poisson random field representation for it. We then analyzed various distributional properties of the SRF $\{S(B),\ B\in\mathcal{B}\}$ for $\mathcal{J}=\{-1,1\}$. Thereafter, we consider the case $M=2$, that is, a GSRF on the positive quadrant of the plane which we denote by $\{\mathcal{S}(s,t),\ (s,t)\in\mathbb{R}^2_+\}$. It is a two parameter L\'evy process with rectangular increments. In Theorem \ref{thm31}, we derive a finite dimensional-type convergence result for it. Moreover, we consider the two-parameter SRF and derive the governing differential equation for its probability generating function. In Section \ref{sec5}, we introduce and study the distributional properties of some integrals of the GSRF over finite rectangles in $\mathbb{R}^2_+$. A scaled compound characterization for its Riemann integral is obtained. In Section \ref{sec4}, we analyze three different fractional variants of the two parameter SRF. The explicit expressions for their point probabilities and the associated governing differential equations are derived. The first two moments and the auto-covariance functions of these processes are obtained. Moreover, it is remarked that these processes are equal in distribution with a particular type of fractional compound Poisson random fields in the sense of \cite{Vishwakarma2025b}.

\section{Preliminaries}\label{sec2} Here, we recall some definitions and some known results that will be used in this paper. 

\subsection{Poisson random field}\label{prfdef}
Let $\mathcal{B}$ be the Borel sigma algebra on $\mathbb{R}^M_+$, $M\ge1$, and let $|B|$ denote the Lebesgue measure of set $B\in\mathcal{B}$. The Poisson point process or Poisson random field on $\mathbb{R}^M$ is a non-negative integer-valued random measure $\{N(B),\ B\in\mathcal{B}\}$, where $N(B)$ denotes the random number of points inside a Borel set $B$. It is also called the spatial Poisson point process, and has the following characterization (see \cite{Stoyan1995}):

\noindent (i) There exist a $\lambda>0$ such that, for $B\subset{\mathbb{R}^M}_+$ with $|B|<\infty$, the random points count $N(B)$ in set $B$ has Poisson distribution with mean $\lambda|B|$, that is,
\begin{equation}\label{dprfdist}
	\mathrm{Pr}\{N(B)=k\}=\frac{e^{-\lambda|B|}(\lambda|B|)^k}{k!},\ k\ge0,
\end{equation}
\noindent (ii) for any finite collection $\{B_1,B_2,\dots,B_m\}$ of disjoint  sets, the random variables $N(B_1)$, $N(B_2)$, $\dots,N(B_m)$ are independent  of each other.

For $B\in\mathcal{B}$ such that $o(|B|)/|B|\to0$ as $|B|\to0$, we have 
\begin{equation}\label{tpprf}
	\mathrm{Pr}\{N(B)=n\}=\begin{cases}
		\lambda|B|+o(|B|),\ n=1,\\
		1-\lambda|B|+o(|B|),\ n=0,\\
		o(|B|),\ n>1.
	\end{cases}
\end{equation}

For $M=2$, let us consider a Poisson random field (PRF) $\{N(s,t), (s,t)\in\mathbb{R}^2_+\}$ on positive quadrant of plane defined by
$
N(s,t)\coloneqq N([0,s]\times[0,t])$, $ (s,t)\in\mathbb{R}_+^2
$, $N(s,t)$ denotes the total number of random points inside the rectangle $[0,s]\times[0,t]$. For $(s,t)\preceq (s',t')$, that is, $s\leq s'$ and $t\leq t'$, the increment of PRF over rectangle $(s,s']\times (t,t']$ is defined as follows:
\begin{equation}\label{recinc}
	\Delta_{s,t}N(s',t')=N((s,s']\times (t,t'])\coloneqq N(s',t')-N(s,t')-N(s',t)+N(s,t).
\end{equation}
The PRF $\{N(s,t),\ (s,t)\in\mathbb{R}^2_+\}$ on $\mathbb{R}^2_+$ is a c\'adl\'ag process such that $N(0,t_2)=N(t_1,0)=0$ with probability one. The increment $\Delta_{s,t}N(s',t')$ is a Poisson random variable with mean $\lambda(s'-s)(t'-t)$ for some constant $\lambda>0$. Its characteristic function is given by
\begin{equation}\label{chprf}
	\mathbb{E}e^{iuN(s,t)}=\exp(\lambda st(e^{iu}-1)),\ u\in\mathbb{R}.
\end{equation}

\subsection{Inverse stable subordinator and fractional Poisson random field}\label{fprf}
Let $\{H^\alpha(t),\ t\ge0\}$, $0<\alpha<1$ be an $\alpha$-stable subordinator. The first passage time process $\{E^\alpha(t),\ t\ge0\}$ defined by 
\begin{equation*}
    E^\alpha(t)=\inf\{\tau>0:H^\alpha(\tau)>t\},\ t\ge0
\end{equation*}
is called an inverse $\alpha$-stable subordinator (see \cite{Meerschaert2013}). We assume that $E^1(t)=t$. Its Laplace transforms with respect to space and times variables are given by $\mathbb{E}e^{-uE^\alpha(t)}=E_{\alpha,1}(-ut^\alpha)$, $u>0$ and
\begin{equation}\label{islap}
	\int_{0}^{\infty}e^{-wt}\mathrm{Pr}\{E^\alpha(t)\in\mathrm{d}x\}\,\mathrm{d}t=w^{\alpha-1}e^{-tw^\alpha},\ w>0,
\end{equation}
respectively. Its mean and variance are
\begin{equation}\label{ismeam}
	\mathbb{E}E^\alpha(t)=\frac{t^\alpha}{\Gamma(\alpha+1)},\ t>0
\end{equation}
and 
\begin{equation*}
	\mathbb{V}\mathrm{ar}E^\alpha(t)=t^{2\alpha}\bigg(\frac{2}{\Gamma(2\alpha+1)}-\frac{1}{\Gamma^2(\alpha+1)}\bigg),\ t>0,
\end{equation*}
respectively.

Let $\{E_1^\alpha(t),\ t\ge0\}$ and $\{E_2^\beta(t),\ t\ge0\}$ be two independent inverse stable subordinators that are independent of the PRF $\{N(s,t),\ (s,t)\in\mathbb{R}^2_+\}$, defined in Section \ref{sec1}. Leonenko and Merzbach \cite{Leonenko2015} introduced and studied a time-changed variant of PRF $\{N(E_1^\alpha(s),E_2^\beta(t)),\ (s,t)\in\mathbb{R}^2_+\}$. Its various characterizations are given in \cite{Aletti2018}. Recently, Kataria and Vishwakarma \cite{Kataria2025}, introduced and studied a fractional variants of the PRF, denoted by $\{N^{\alpha,\beta}(s,t),\ (s,t)\in\mathbb{R}^2_+\}$, $0<\alpha,\beta\leq1$, whose distribution $q^{\alpha,\beta}(n,s,t)=\mathrm{Pr}\{N^{\alpha,\beta}(s,t)=n\}$, $n\ge0$ solves the following system of  partial fractional differential equations:
\begin{equation*}
	\frac{\partial^{\alpha+\beta}}{\partial t^\beta\partial s^\alpha}q^{\alpha,\beta}(n,s,t)=\lambda(n+1)q^{\alpha,\beta}(n+1,s,t)-\lambda(2n+1)q^{\alpha,\beta}(n,s,t)+\lambda nq^{\alpha,\beta}(n-1,s,t),
\end{equation*}
with initial conditions $q^{\alpha,\beta}(0,0,t)=q^{\alpha,\beta}(0,s,0)=1$.  Here, the derivative involve is the Caputo fractional derivative defined as follows (see \cite{Kilbas2006}):
\begin{equation}\label{cder}
	\frac{\mathrm{d}^\alpha}{\mathrm{d}x^\alpha}f(x)=\begin{cases}
		\frac{1}{\Gamma(1-\alpha)}\int_{0}^{x}\frac{\frac{\mathrm{d}}{\mathrm{d}y}f(y)}{(x-y)^\alpha}\,\mathrm{d}y,\ 0<\alpha<1,\\
		\frac{\mathrm{d}}{\mathrm{d}x}f(x),\ \alpha=1.
	\end{cases}
\end{equation}
It is called the fractional Poisson random field (FPRF). Also, it is established that $N^{\alpha,\beta}(s,t)\overset{d}{=}N(E_1^\alpha(s),E_2^\beta(t))$ for $0<\alpha,\beta<1$, where $\overset{d}{=}$ denotes the equality in distribution. The distribution of FPRF is given by
\begin{equation}\label{fprfdist}
	q^{\alpha,\beta}(n,s,t)=\sum_{k=n}^{\infty}\frac{(-1)^{k-n}k_{(k-n)}k_{(n)}(\lambda s^\alpha t^\beta)^k}{\Gamma(k\alpha+1)\Gamma(k\beta+1)},\ n\ge0,
\end{equation}
where $k_{(n)}=k(k-1)\dots(k-n+1)$ with $k_{(0)}=1$.

For $(s,t)$ and $(s',t')$ in $\mathbb{R}^2_+$, the mean, variance and the auto-covariance of the FPRF are given by
\begin{equation}\label{fprfmean}
	\mathbb{E}N^{\alpha,\beta}(s,t)=\frac{\lambda s^\alpha t^\beta}{\Gamma(\alpha+1)\Gamma(\beta+1)},
\end{equation}
\begin{equation}\label{fprfvar}
	\mathbb{V}\mathrm{ar}N^{\alpha,\beta}(s,t)=\frac{\lambda s^\alpha t^\beta}{\Gamma(\alpha+1)\Gamma(\beta+1)}+\frac{(2\lambda s^\alpha t^\beta)^2}{\Gamma(2\alpha+1)\Gamma(2\beta+1)}-\frac{(\lambda s^\alpha t^\beta)^2}{\Gamma^2(\alpha+1)\Gamma^2(\beta+1)}
\end{equation}
and
\begin{align}\label{fprfcov}
	\mathbb{C}\mathrm{ov}(N^{\alpha,\beta}(s,t),N^{\alpha,\beta}(s',t'))&=\frac{\lambda(s\wedge s')^\alpha(t\wedge t')^\beta}{\Gamma(\alpha+1)\Gamma(\beta+1)}-\frac{\lambda^2(ss')^\alpha(tt')^\beta}{\Gamma^2(\alpha+1)\Gamma^2(\beta+1)}\nonumber\\
	&\ \ +\frac{\lambda}{\alpha\Gamma^2(\alpha)}\int_{0}^{s\wedge s'}((s-x)^\alpha+(s'-x)^\alpha)x^{\alpha-1}\,\mathrm{d}x\nonumber\\
	&\ \ \cdot\frac{\lambda}{\beta\Gamma^2(\beta)}\int_{0}^{t\wedge t'}((t-y)^\beta+(t'-y)^\beta)y^{\beta-1}\,\mathrm{d}y,
\end{align}
respectively. Here, $s\wedge t$ denotes the minimum of $s$ and $t$.
\section{Skellam random fields}\label{sec3}
In this section, we introduce a generalized Skellam point process on $\mathbb{R}^M_+$. Let $\mathcal{J}\subset\mathbb{R}-\{0\}$ be a finite subset. Let $\{N_j(B),\ B\in\mathcal{B}\}$, $j\in\mathcal{J}$ be independent Poisson random fields on $\mathbb{R}^M_+$ with parameters $\lambda_j>0$, $j\in\mathcal{J}$, respectively. We consider a random field $\{\mathcal{S}(B),\ B\in\mathcal{B}\}$ on $\mathbb{R}^M_+$ defined as follows:
\begin{equation}\label{gsrf}
	\mathcal{S}(B)\coloneqq\sum_{j\in \mathcal{J}}jN_j(B).
\end{equation}
We call it the generalized Skellam random field (GSRF) with parameters $\{\lambda_j\}_{j\in\mathcal{J}}$. 
For $M=1$ and considering the non-negative positive index only, it reduces to a generalized Skellam point process $\{\sum_{j\in\mathcal{J}}jN_j(t),\ t\ge0\}$, where $\{N_j(t),\ t\ge0\}$'s are independent Poisson processes. It was introduced and studied in \cite{Cinque2025}. Moreover, for $\mathcal{ J}=\{1\}$, the GSRF reduces to the PRF on $\mathbb{R}^M_+$.

For $B\in\mathcal{B}$ such that $|B|<\infty$, the mean and variance of (\ref{gsrf}) are given by $\mathbb{E}\mathcal{S}(B)=\sum_{j\in\mathcal{J}}j\lambda_j|B|$ and $\mathbb{V}\mathrm{ar}\mathcal{S}(B)=\sum_{j\in\mathcal{J}}j^2\lambda_j|B|$, respectively. For $B_1$ and $B_2$ in $\mathcal{B}$, the covariance of PRF $\{N_j(B),\ B\in\mathcal{B}\}$ is given by $\mathbb{C}\mathrm{ov}(N_j(B_1),N_j(B_2))=\lambda_j|B_1\cap B_2|$ for each $j\in\mathcal{J}$ (see \cite{Stoyan1995}, Eq. (2.23)). Thus, the auto covariance of GSRF is given by $\mathbb{C}\mathrm{ov}(\mathcal{S}(B_1),\mathcal{S}(B_2))=\sum_{j\in\mathcal{J}}j^2\lambda_j|B_1\cap B_2|$.
\begin{remark}
	Whenever $\mathcal{J}$ is a collection of finitely many integers, for any $B\in\mathcal{B}$ such that $o(|B|)/|B|\to0$ as $|B|\to0$, from (\ref{tpprf}), we have
	\begin{equation*}
		\mathrm{Pr}\{\mathcal{S}(B)=j\}=\begin{cases}
			\lambda_j|B|+o(|B|),\ j\in\mathcal{J},\\
			1-\sum_{j\in\mathcal{J}}\lambda_j|B|+o(|B|),\ j=0,\\
			o(|B|),\ \text{otherwise}.
		\end{cases}
	\end{equation*}	
\end{remark}
\begin{proposition}\label{gsrfcrep}
	Let $\{\mathcal{Y}_k\}_{k\ge1}$ be a sequence of independent and identically distributed (iid) random variables taking value in $\mathcal{ J}$ such that $\mathrm{Pr}\{\mathcal{Y}_1=j\}=\lambda_j/\sum_{j\in\mathcal{J}}\lambda_j$, $j\in\mathcal{J}$. Also, let $\{N(B),\ B\in\mathcal{B}\}$ be a PRF on $\mathbb{R}^M_+$ with parameter $\sum_{j\in\mathcal{J}}\lambda_j$ that is independent of $\mathcal{Y}_k$'s. Then, the GSRF with parameters $\{\lambda_j\}_{j\in\mathcal{J}}$ satisfy the following equality:
	\begin{equation}
		\mathcal{S}(B)\overset{d}{=}\sum_{k=1}^{N(B)}\mathcal{Y}_k,\ B\in\mathcal{B},
	\end{equation}
	where $\overset{d}{=}$ denotes the equality in distribution.
\end{proposition}
\begin{proof}
	The moment generating function (mgf) of $\mathcal{Y}_1$ is given by
	$
	\mathbb{E}\exp(u\mathcal{Y}_1)={\sum_{j\in\mathcal{J}}e^{uj}\lambda_j}/{\sum_{j\in\mathcal{J}}\lambda_j}$, $ u\in\mathbb{R}$.
	So,
	\begin{align}
		\mathbb{E}\exp\bigg(u\sum_{k=1}^{N(B)}\mathcal{Y}_k\bigg)=\mathbb{E}(\mathbb{E}\exp(u\mathcal{Y}_1))^{N(B)}&=\mathbb{E}\bigg(\frac{\sum_{j\in\mathcal{J}}e^{uj}\lambda_j}{\sum_{j\in\mathcal{J}}\lambda_j}\bigg)^{N(B)}\nonumber\\
		&=\exp\bigg(\sum_{j\in\mathcal{J}}\lambda_j|B|\bigg(\frac{\sum_{j\in\mathcal{J}}e^{uj}\lambda_j}{\sum_{j\in\mathcal{J}}\lambda_j}-1\bigg)\bigg).\label{gspf11}
	\end{align}
	
	The mgf of $\mathcal{S}(B)$ is given by
	\begin{align*}
		\mathbb{E}\exp(u\mathcal{S}(B))=\mathbb{E}\exp\bigg(u\sum_{j\in\mathcal{J}}jN_j(B)\bigg)&=\prod_{j\in\mathcal{J}}\mathbb{E}\exp(ujN_j(B))\\
		&=\prod_{j\in\mathcal{J}}\exp(\lambda_j|B|(e^{uj}-1))\\
		&=\exp\bigg(\sum_{j\in\mathcal{J}}\lambda_j|B|(e^{uj}-1)\bigg),
	\end{align*}
	which coincides with (\ref{gspf11}). This completes the proof.
\end{proof}
\begin{remark}
	Let $\{N(B),\ B\in\mathcal{B}\}$ and $\{N'(B),\ B\in\mathcal{B}\}$ be independent PRFs with parameters $\lambda>0$ and $\lambda'>0$, respectively. Then, $\{N(B)+N'(B),\ B\in\mathcal{B}\}$ is a PRF with parameter $\lambda'+\lambda'$. Therefore, for independent GSRFs $\{S(B),\ B\in\mathcal{B}\}$ and $\{S'(B),\ B\in\mathcal{B}\}$ with parameters $\{\lambda_j\}_{j\in\mathcal{J}}$ and $\{\lambda'_j\}_{j\in\mathcal{J}}$, respectively, the random field $\{S(B)+S'(B),\ B\in\mathcal{B}\}$ is a GSRF with parameters $\{\lambda_j+\lambda'_j\}_{j\in\mathcal{J}}$.
\end{remark}
\subsection{Skellam random field}
Now, we consider the special case of  $\mathcal{J}=\{1,-1\}$ in (\ref{gsrf}).
Let $\{N_1(B),\ B\in\mathcal{B}\}$ and $\{N_2(B),\ B\in\mathcal{B}\}$ be two independent Poisson random fields on $\mathbb{R}^M_+$ with parameter $\lambda_1>0$ and $\lambda_2>0$, respectively. We consider a random field $\{S(B),\ B\in\mathcal{B}\}$ defined as follows: 
\begin{equation}\label{SRF0}
	S(B)\coloneqq N_1(B)-N_2(B).
\end{equation}
We call it the Skellam random field (SRF). Its mean and variance are given by $\mathbb{E}S(B)=(\lambda_1-\lambda_2)|B|$ and $\mathbb{V}\mathrm{ar}S(B)=(\lambda_1+\lambda_2)|B|$, respectively. Its auto covariance is $\mathbb{C}\mathrm{ov}(S(B_1),S(B_2))=(\lambda_1+\lambda_2)|B_1\cap B_2|$. 

The probability generating function (pgf) of (\ref{SRF0}) is given by
\begin{equation}\label{pgf1}
	\mathbb{E}u^{S(B)}=\mathbb{E}u^{N_1(B)}\mathbb{E}(1/u)^{N_2(B)}=\exp\bigg(\lambda_1|B|(u-1)+\lambda_2|B|\bigg(\frac{1}{u}-1\bigg)\bigg),\ 0<u\leq1,
\end{equation}
where we have used $\mathbb{E}u^{N_i(B)}=\exp(\lambda_i|B|(u-1))$, $0<u\leq1$, $i=1,2$, the pgf of PRF. Also, its mgf is
\begin{equation}\label{srfmgf}
	\mathbb{E}e^{\xi S(B)}=\exp\big(\lambda_1|B|(e^{\xi}-1)+\lambda_2|B|(e^{-\xi}-1)\big),\ \xi\in\mathbb{R}.
\end{equation}

For $n\ge0$, its point probabilities $p(n,B)=\mathrm{Pr}\{S(B)=n\}$ are given by
\begin{align*}
	p(n,B)&=\sum_{k=0}^{\infty}\mathrm{Pr}\{N_1(B)=n+k\}\mathrm{Pr}\{N_2(B)=k\}\\
	&=\sum_{k=0}^{\infty}e^{-\lambda_1|B|}\frac{(\lambda_1|B|)^{n+k}}{(n+k)!}e^{-\lambda_2|B|}\frac{(\lambda_2|B|)^k}{k!}\\
	&=e^{-(\lambda_1+\lambda_2)|B|}\bigg(\frac{\lambda_1}{\lambda_2}\bigg)^{n/2}I_n(2\sqrt{\lambda_1\lambda_2}|B|), 
\end{align*}
where $I_n$ is the modified Bessel function defined in (\ref{Besselm}). Thus, by the symmetry, for any $n\in\mathbb{Z}$, it is given by
\begin{equation*}
	p(n,B)=e^{-(\lambda_1+\lambda_2)|B|}\bigg(\frac{\lambda_1}{\lambda_2}\bigg)^{n/2}I_{|n|}(2\sqrt{\lambda_1\lambda_2}|B|).
\end{equation*}
\begin{remark}
	For $M=1$, and considering only the non-negative real line as indexing set, the process as defined in (\ref{SRF0}) gives the standard Skellam process on the non-negative real line,  $\{S(t),\ t\ge0\}$. Its one dimensional distribution is given as follows:
	\begin{equation*}
		\mathrm{Pr}\{S(t)=n\}=e^{-(\lambda_1+\lambda_2)t}\bigg(\frac{\lambda_1}{\lambda_2}\bigg)^{n/2}I_{|n|}(2\sqrt{\lambda_1\lambda_2}t),\ n\in\mathbb{Z}.
	\end{equation*}
	Moreover, its mean, variance and auto covariance are given by $\mathbb{E}S(t)=(\lambda_1-\lambda_2)t$, $\mathbb{V}\mathrm{ar}S(t)=(\lambda_1+\lambda_2)t$ and $\mathbb{C}\mathrm{ov}(S(t),S(t'))=(\lambda_1+\lambda_2)(t\wedge t')$ for all $t>0$ and $t'>0$, respectively.
\end{remark} 
\begin{remark}\label{rem31}
	For $B\in\mathcal{B}$ such that $o(|B|)/|B|\to0$ as $|B|\to0$, we have
	\begin{equation*}
		\mathrm{Pr}\{S(B)=n\}=\sum_{k=0}^{\infty}\mathrm{Pr}\{N_1(B)=n+k\}\mathrm{Pr}\{N_2(B)=k\},\ n\ge0.
	\end{equation*}
	From (\ref{tpprf}), it follows that
	\begin{equation*}
		\mathrm{Pr}\{S(B)=n\}=\begin{cases}
			\lambda_1|B|+o(|B|),\ n=1,\\
			\lambda_2|B|+o(|B|),\ n=-1,\\
			1-(\lambda_1-\lambda_2)|B|,\ n=0,\\
			o(|B|),\ \text{otherwise}.
		\end{cases}
	\end{equation*}
\end{remark}
\begin{remark}\label{rem3.4}
	Let $\{S(B),\ B\in\mathcal{B}\}$ be a random field as defined in (\ref{SRF0}) and $\lambda_i>0$ be rate of $N_i(B)$, $i=1,2$. Also, let $Y_1,Y_2,Y_3,\dots$ be iid discrete random variables such that $
		\mathrm{Pr}\{Y_1=1\}={\lambda_1}/{(\lambda_1+\lambda_2)}$ and $			\mathrm{Pr}\{Y_1=-1\}={\lambda_2}/{(\lambda_1+\lambda_2)}$.
	Then, $
		S(B)\overset{d}{=}\sum_{j=1}^{N(B)}Y_j$,
	where $\{N(B),\ B\in\mathcal{B}\}$ is a PRF on $\mathbb{R}^M_+$ with parameter $\lambda_1+\lambda_2$ and it is assumed to be independent of $Y_j$'s. 
\end{remark}
\subsection{GSRF on $\mathbb{R}^2_+$}  For $M=2$, the GSRF defined in (\ref{gsrf}) reduces to the two parameter random field $\{\mathcal{S}(s,t),\ (s,t)\in\mathbb{R}^2_+\}$ defined as follows:

\begin{equation}\label{gsrf2}
	\mathcal{S}(s,t)\coloneqq \sum_{j\in\mathcal{J}}jN_j(s,t),
\end{equation}
where $\{N_j(s,t),\ (s,t)\in\mathbb{R}^2_+\}$, $j\in\mathcal{J}$ are independent PRFs on plane with parameters $\lambda_j>0$, $j\in\mathcal{ J}$, respectively. For $(s,t)$ and $(s',t')$ in $\mathbb{R}^2_+$ such that $(s,t)\preceq(s',t')$, the rectangular increment of GSRF defined by (\ref{gsrf2}) is given by
\begin{align*}
	\Delta_{s,t}\mathcal{S}(s',t')&=\mathcal{S}(s',t')-\mathcal{S}(s',t)-\mathcal{S}(s,t')+\mathcal{S}(s,t)\\
	&=\sum_{j\in\mathcal{J}}jN_j(s',t')-\sum_{j\in\mathcal{J}}jN_j(s',t)-\sum_{j\in\mathcal{J}}jN_j(s,t')+\sum_{j\in\mathcal{J}}jN_j(s,t)\\
	&=\sum_{j\in\mathcal{J}}j\Delta_{s,t}N_j(s',t').
\end{align*}
Thus, the GSRF $\{\mathcal{S}(s,t),\ (s,t)\in\mathbb{R}^2_+\}$ has independent and stationary rectangular increments. It follows from the independent and stationary rectangular increments of $\{N_j(s,t),\ (s,t)\in\mathbb{R}^2_+\}$, $j\in\mathcal{J}$ and their independence.

%
\begin{theorem}\label{thm31}	For $l\ge1$, $l'\ge1$ and $k\ge1$, let $p_{l,l',j}^{(k)}\in(0,1)$ for all $j\in\mathcal{J}$ such that $\sum_{j\in\mathcal{J}}p_{l,l',j}^{(k)}<1$. Also, let $\{\mathcal{X}_l^{(k)}\}_{l\ge1,k\ge0}$ be independent random variables such that
	\begin{equation}\label{xdef}
		\mathcal{X}_{l,l'}^{(k)}=\begin{cases}
			j\in\mathcal{J},\ \text{with probability $p_{l,l',j}^{(k)}$},\\
			0,\ \text{with probability $1-\sum_{j\in\mathcal{J}}p_{l,l',j}^{(k)}$}.
		\end{cases}
	\end{equation} 
	Let us consider a random field
	$\{\mathcal{Z}_k(s,t),\ (s,t)\in\mathbb{R}^2_+\}$ on plane defined as follows: 
	\begin{equation}\label{zkdef}
		\mathcal{Z}_k(s,t)=\sum_{l=1}^{[ks]}\sum_{l'=1}^{[kt]}\mathcal{X}_{l,l'}^{(k)},\ (s,t)\in\mathbb{R}^2_+.
	\end{equation}
	
	If 
	\begin{equation}\label{limas}
		\sum_{l=1}^{[ks]}\sum_{l'=1}^{[kt]}p_{l,l',j}^{(k)}\longrightarrow\lambda_jst,\ (s,t)\in\mathbb{R}^2_+\ \ \text{and}\ \ \max_{0\prec(l,l')\preceq (k,k)}p_{l,l',j}^{(k)}\longrightarrow0\ \ \text{as $k\longrightarrow\infty$},
	\end{equation}
	then for $(s_r,t_r)$, $r=1,\dots,m$ in $\mathbb{R}^2_+$ such that $(s_1,t_1)\preceq(s_2,t_2)\preceq\dots\preceq(s_m,t_m)$, $m\ge1$, we have
	\begin{equation}\label{cng}
		(\mathcal{Z}_k(s_1,t_1), \mathcal{Z}_k(s_2,t_2),\dots,\mathcal{Z}_k(s_m,t_m))\overset{d}{\longrightarrow}(\mathcal{S}(s_1,t_1), \mathcal{S}(s_2,t_2),\dots,\mathcal{S}(s_m,t_m))\ \ \text{as $k\longrightarrow\infty$},
	\end{equation}
	where $\overset{d}{\longrightarrow}$ denotes the convergence in distribution.
\end{theorem} 
\begin{proof}
	To show the convergence (\ref{cng}), it is enough to show that for real constants $c_1,c_2,\dots c_m$, the random variable $\sum_{r=1}^{m}c_r\mathcal{Z}_k(s_r,t_r)\overset{d}{\longrightarrow}\sum_{r=1}^{m}c_r\mathcal{S}(s_r,t_r)$ as $k\longrightarrow\infty$. Let $(s_0,t_0)=(0,0)$ and $\mathcal{Z}_k(0,0)=0$. Then, $\sum_{r=1}^{m}c_r\mathcal{Z}_k(s_r,t_r)$ can be rewritten as
	\begin{align*}
		\sum_{r=1}^{m}c_r\mathcal{Z}_k(s_r,t_r)
		&=\sum_{r'=0}^{m-1}\sum_{r=1}^{m-r'}c_{r+r'}\{\mathcal{Z}_k(s_r,t_r)-\mathcal{Z}_k(s_{r-1},t_{r-1})\}\\
		&=\sum_{r=1}^{m}\sum_{r'=0}^{m-r}c_{r+r'}\{\mathcal{Z}_k(s_r,t_r)-\mathcal{Z}_k(s_{r-1},t_{r-1})\}.
	\end{align*}
	 For $(s_{r-1},t_{r-1})\preceq(s_r,t_r)$, $1\leq r\leq m$, we note that random variables $\mathcal{Z}_k(s_2,t_2)-\mathcal{Z}_k(s_1,t_1), \mathcal{Z}_k(s_3,t_3)-\mathcal{Z}_k(s_2,t_2),\dots,\mathcal{Z}_k(s_m,t_m)-\mathcal{Z}_k(s_{m-1},t_{m-1})$ are mutually independent, as they are functions of increments over disjoint rectangles. For example, from Figure \ref{fig1}, it follows that $\mathcal{Z}_k(s_2,t_2)-\mathcal{Z}_k(s_1,t_1)$ depends on the increments over rectangles $\mathrm{I}, \mathrm{II}$ and $\mathrm{III}$, and $\mathcal{Z}_k(s_3,t_3)-\mathcal{Z}_k(s_2,t_2)$ depends on the increments over rectangles $\mathrm{IV}, \mathrm{V}$ and $\mathrm{VI}$. 
\begin{figure}[ht!]
	\includegraphics[width=6cm]{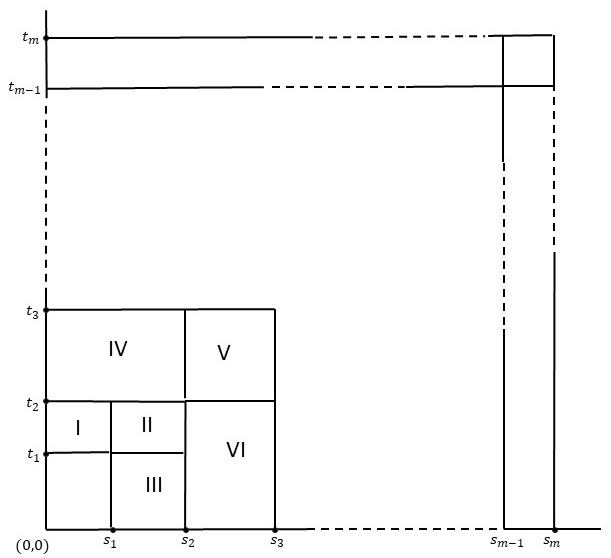}
	\caption{Increments of $\mathcal{Z}_k$}\label{fig1}
\end{figure}
Moreover, for any $(s,t)\prec(s',t')$, the rectangular increment of $\{\mathcal{Z}_k(s,t),\ (s,t)\in\mathbb{R}^2_+\}$ is given by
\begin{align}
	\Delta_{s,t}\mathcal{Z}_k(s',t')&=\mathcal{Z}_k(s',t')-\mathcal{Z}_k(s',t)-\mathcal{Z}_k(s,t')+\mathcal{Z}_k(s,t)\nonumber\\
	&=\sum_{l=1}^{[ks']}\sum_{l'=1}^{[kt']}\mathcal{X}_{l,l'}^{(k)}-\sum_{l=1}^{[ks']}\sum_{l'=1}^{[kt]}\mathcal{X}_{l,l'}^{(k)}-\sum_{l=1}^{[ks]}\sum_{l'=1}^{[kt']}\mathcal{X}_{l,l'}^{(k)}+\sum_{l=1}^{[ks]}\sum_{l'=1}^{[kt]}\mathcal{X}_{l,l'}^{(k)}\nonumber\\
	&=\sum_{l=1}^{[ks']}\sum_{l'=[kt]+1}^{[kt']}\mathcal{X}_{l,l'}^{(k)}-\sum_{l=1}^{[ks]}\sum_{l'=[kt]+1}^{[kt']}\mathcal{X}_{l,l'}^{(k)}\nonumber\\
	&=\sum_{l=[ks]+1}^{[ks']}\sum_{l'=[kt]+1}^{[kt']}\mathcal{X}_{l,l'}^{(k)}.\label{zkinc}
\end{align}
Thus, from independence of $\mathcal{X}_{l,l'}^{(k)}$'s, it follows that the random field defined by (\ref{zkdef}) has independent rectangular increments. Hence, to prove (\ref{cng}), it is enough to show that for $(s,t)\preceq(s',t')$, the following convergence holds: 
\begin{equation}\label{lime1}
	\mathcal{Z}_k(s',t')-\mathcal{Z}_k(s,t)\overset{d}{\longrightarrow} \mathcal{S}(s',t')-\mathcal{S}(s,t)\ \text{as}\ k\longrightarrow\infty.
\end{equation} 

For $(s,t)\preceq(s',t')$, we have
\begin{align*}
	\mathcal{Z}_k(s',t')-\mathcal{Z}_k(s,t)&=\mathcal{Z}_k(s',t')-\mathcal{Z}_k(s',t)-\mathcal{Z}_k(s,t')+\mathcal{Z}_k(s,t)-\mathcal{Z}_k(s,t)+\mathcal{Z}_k(s',t)+\mathcal{Z}_k(s,t')\nonumber\\
	&=\Delta_{s,t}\mathcal{Z}_k(s',t')-\Delta_{0,t}\mathcal{Z}_k(s,t')-\Delta_{s,0}\mathcal{Z}_k(s',t).
\end{align*}
Hence, in view of the independent rectangular increments property of $\{\mathcal{Z}_k(s,t),\ (s,t)\in\mathbb{R}^2_+\}$, to show (\ref{lime1}), it is sufficient to show $\Delta_{s,t}\mathcal{Z}_k(s',t')\overset{d}{\longrightarrow}\Delta_{s,t}\mathcal{S}(s',t')$ as $k\longrightarrow\infty$.

From (\ref{xdef}) and (\ref{zkinc}), the Fourier transform of $\Delta_{s,t}\mathcal{Z}_k(s',t')$ is given by
\begin{align*}
	\mathbb{E}\exp(iu\Delta_{s,t}\mathcal{Z}_k(s',t'))&=\prod_{l=[ks]+1}^{[ks']}\prod_{l'=[kt]+1}^{[kt']}\mathbb{E}e^{iu\mathcal{X}_{l,l'}^{(k)}},\ u\in\mathbb{R},\nonumber\\
	&=\prod_{l=[ks]+1}^{[ks']}\prod_{l'=[kt]+1}^{[kt']}\bigg(\sum_{j\in\mathcal{J}}e^{iuj}p_{l.l',j}^{(k)}+1-\sum_{j\in \mathcal{J}}p_{l,l',j}^{(k)}\bigg)\nonumber\\
	&=\exp\bigg(\sum_{l=[ks]+1}^{[ks']}\sum_{l'=[kt]+1}^{[kt']}\ln\bigg(\sum_{j\in\mathcal{J}}(e^{iuj}-1)p_{l,l',j}^{(k)}+1\bigg)\bigg)\nonumber\\
	&\sim \exp\bigg(\sum_{l=[ks]+1}^{[ks']}\sum_{l'=[kt]+1}^{[kt']}\sum_{j\in\mathcal{J}}(e^{iuj}-1)p_{l,l',j}^{(k)}\bigg),
\end{align*}
where the last step follows by using the assumption $p_{l,l',j}^{(k)}\longrightarrow0$ as $k\longrightarrow0$, and the approximation $\ln(1+x)\sim x$ as $x\to 0$. Thus, by using (\ref{limas}), we get
\begin{equation}\label{limchf}
	\lim_{k\longrightarrow\infty}\mathbb{E}\exp(iu\Delta_{s,t}\mathcal{Z}_k(s',t'))=\exp\bigg((s'-s)(t'-t)\sum_{j\in \mathcal{J}}(e^{iuj}-1)\lambda_j\bigg),\ u\in\mathbb{R}.
\end{equation}

Now, by using the stationary rectangular increments property of GSRF, the Fourier transform of $\Delta_{s,t}\mathcal{S}(s',t')$ is given by
\begin{align*}
	\mathbb{E}\exp(iu\Delta_{s,t}\mathcal{S}(s',t'))&=\mathbb{E}\exp(iu\mathcal{S}(s'-s,t'-t))\\
	&=\mathbb{E}\exp\bigg(iu\sum_{j\in \mathcal{J}}jN_j(s'-s,t'-t)\bigg)\\
	&=\prod_{j\in\mathcal{ J}}\mathbb{E}\exp(iujN_j(s'-s,t'-t))\\
	&=\prod_{j\in\mathcal{ J}}\exp((s'-s)(t'-t)\lambda_j(e^{iuj}-1)),\ u\in\mathbb{R},
\end{align*}
which coincides with (\ref{limchf}). This completes the proof.
\end{proof}
\begin{remark}
	 Note that for any $(s,t)$ and $(s',t')$ in $\mathbb{R}^2_+$, we can have following four possible cases: $s\leq s'$, $t\leq t'$ or $s\leq s'$, $t\ge t'$ or $s\ge s'$, $t\leq t'$ or $s\ge s'$, $t\ge t'$. So, it can be shown that the convergence (\ref{lime1}) holds for any $(s,t)$ and $(s',t')$ in $\mathbb{R}^2_+$.
	
	It is sufficient to consider the first two cases only. Case $s\leq s'$, $t\leq t'$ follows from the proof of Theorem \ref{thm31}. If $s\leq s'$, $t\ge t'$ then $\mathcal{Z}_k(s',t')-\mathcal{Z}_k(s,t)=\Delta_{s,0}\mathcal{Z}_k(s',t')-\Delta_{0,t'}\mathcal{Z}_k(s,t)$, which proves the claim. Therefore, under the assumptions of Theorem \ref{thm31}, the convergence (\ref{cng}) holds for all points $(s_1,t_1),\dots,(s_m,t_m)$ in $\mathbb{R}^2_+$ such that $\mathcal{Z}_k(s_2,t_2)-\mathcal{Z}_k(s_1,t_1),\dots,\mathcal{Z}_k(s_m,t_m)-\mathcal{Z}_k(s_{m-1},t_{m-1})$ are mutually independent.
\end{remark}

For $M=2$, the random field defined in (\ref{SRF0}) reduces to the SRF $\{S(s,t),\ (s,t)\in\mathbb{R}^2_+\}$ on $\mathbb{R}^2_+$ defined as follows:
\begin{equation}\label{srf1}
	S(s,t)\coloneqq N_1(s,t)-N_2(s,t),
\end{equation} 
where $\{N_1(s,t),\ (s,t)\in\mathbb{R}^2_+\}$ and $\{N_2(s,t),\ (s,t)\in\mathbb{R}^2_+\}$ are independent Poisson random fields on positive plane with parameter $\lambda_1>0$ and $\lambda_2>0$, respectively. Its distribution $p(n,s,t)=\mathrm{Pr}\{S(s,t)=n\}$ is given by
\begin{equation}\label{srfdist}
	p(n,s,t)=e^{-(\lambda_1+\lambda_2)st}\bigg(\frac{\lambda_1}{\lambda_2}\bigg)^{n/2}I_{|n|}(2\sqrt{\lambda_1\lambda_2}st),\ n\in\mathbb{Z}.
\end{equation}
Its mean, variance and covariance are  $\mathbb{E}S(s,t)=(\lambda_1-\lambda_2)st$,  $\mathbb{V}\mathrm{ar}S(s,t)=(\lambda_1+\lambda_2)st$ and
$\mathbb{C}\mathrm{ov}(S(s,t),S(s',t'))=(\lambda_1+\lambda_2)(s\wedge s')(t\wedge t')$ for all $(s,t)$ and $(s',t')$ in $\mathbb{R}^2_+$.
\begin{remark}\label{rem37}
	The two parameter SRF defined in (\ref{srf1}) satisfies the following equality:
	$
		S(s,t)\overset{d}{=}\sum_{j=1}^{N(s,t)}Y_j,
	$
	where $Y_j$'s are random variables as defined in Remark \ref{rem3.4} and $\{N(s,t),\ (s,t)\in\mathbb{R}^2_+\}$ is a PRF with parameter $\lambda_1+\lambda_2$ which is independent of $Y_j$'s.
\end{remark}
\begin{remark}
	On taking the derivative with respect to $s$ on both sides of (\ref{srfdist}) and using the following result:
	\begin{equation*}
		\frac{\partial}{\partial x}I_{\nu}(x)=\frac{I_{\nu-1}(x)+I_{\nu+1}(x)}{2},
	\end{equation*}
	we get the system of governing differential equations for distribution (\ref{srfdist}) as follows:
	\begin{equation}\label{difeq1}
		\frac{\partial}{\partial s}p(n,s,t)=-(\lambda_1+\lambda_2)tp(n,s,t)+\lambda_1 tp(n-1,s,t)+\lambda_2 tp(n+1,s,t),\ n\in\mathbb{Z},
	\end{equation}
	with $p(0,0,t)=1$ and $p(n,0,t)=0$ for $n\ne0$ and for all $t\ge0$. Similarly, it satisfies
	\begin{equation*}
		\frac{\partial}{\partial t}p(n,s,t)=-(\lambda_1+\lambda_2)sp(n,s,t)+\lambda_1 sp(n-1,s,t)+\lambda_2 sp(n+1,s,t),\ n\in\mathbb{Z}.
	\end{equation*} 
	
	From (\ref{pgf1}), it follows that the pgf $G(u,s,t)=\mathbb{E}u^{S(s,t)}$, $0<u\leq1$ of SRF is
	\begin{equation}\label{srfpgf}
		G(u,s,t)=\exp\bigg(\lambda_1st(u-1)+\lambda_2st\bigg(\frac{1}{u}-1\bigg)\bigg).
	\end{equation}
	On differentiating it with respect to $s$, we get
	\begin{equation*}
		\frac{\partial}{\partial s}G(u,s,t)=t\bigg(\lambda_1(u-1)+\lambda_2\bigg(\frac{1}{u}-1\bigg)\bigg)\exp\bigg(\lambda_1st(u-1)+\lambda_2st\bigg(\frac{1}{u}-1\bigg)\bigg),
	\end{equation*}
	which again on differentiating with respect to $t$ yields
	\begin{align*}
		\frac{\partial^2}{\partial t\partial s}G(u,s,t)&=\bigg(\lambda_1(u-1)+\lambda_2\bigg(\frac{1}{u}-1\bigg)\bigg)G(u,s,t)\\
		&\ \ +st\bigg(\lambda_1(u-1)+\lambda_2\bigg(\frac{1}{u}-1\bigg)\bigg)^2\exp\bigg(\lambda_1st(u-1)+\lambda_2st\bigg(\frac{1}{u}-1\bigg)\bigg). 
	\end{align*}
	Thus, by using 
	\begin{equation*}
		\frac{\partial}{\partial u}G(u,s,t)=st\bigg(\frac{\lambda_1u^2-\lambda_2}{u^2}\bigg)G(u,s,t),
	\end{equation*}
	for all $0<u\leq1$ such that $\lambda_1 u^2-\lambda_2\ne0$, we get the following governing partial differential equation for pgf of SRF:
	\begin{equation}\label{pgf2}
		\frac{\partial^2}{\partial t\partial s}G(u,s,t)=\bigg(\lambda_1(u-1)+\lambda_2\bigg(\frac{1}{u}-1\bigg)\bigg)G(u,s,t)+\frac{(\lambda_1(u-1)u+\lambda_2(1-u))^2}{\lambda_1 u^2-\lambda_2}\frac{\partial}{\partial u}G(u,s,t),
	\end{equation}
	with initial condition $G(u,0,t)=G(u,s,0)=1$.
\end{remark}
\section{Integrals of SRF}\label{sec5} The Riemann-Liouville fractional integral of one parameter generalized Skellam process is introduced and studied in \cite{Cinque2025}. Here, we study integrals of GSRF over rectangles in plane. First, we recall the definition of Riemann-Liouville fractional integral.
\begin{definition}
	For an integrable function $f$, its Riemann-Liouville intergal is defined by
	\begin{equation*}
		I^\nu_t f(t)=\frac{1}{\Gamma(\nu)}\int_{0}^{t}\frac{f(s)}{(t-s)^{1-\nu}}\,\mathrm{d}s,\ \nu>0.
	\end{equation*}
\end{definition}
Let $\{N(s,t),\ (s,t)\in\mathbb{R}^2_+\}$ be the PRF with parameter $\lambda>0$. Its Riemann-Liouville fractional integral is defined as follows:
\begin{equation}\label{frint}
	X^{\nu_1,\nu_2}(s,t)=\frac{1}{\Gamma(\nu_1)\Gamma(\nu_2)}\int_{0}^{s}\int_{0}^{t}(s-x)^{\nu_1-1}(t-y)^{\nu_2-1}N(x,y)\,\mathrm{d}x\,\mathrm{d}y,\ \nu_i>0,\ i=1,2,\ (s,t)\in\mathbb{R}^2_+.
\end{equation}
It was studied in \cite{Kataria2025}. The mean and variance of $X^{\nu_1,\nu_2}(s,t)$ are given by 
\begin{equation*}
	\mathbb{E}X^{\nu_1,\nu_2}(s,t)=\frac{\lambda s^{\nu_1+1}t^{\nu_2+1}}{\Gamma(\nu_1+2)\Gamma(\nu_2+2)}
\end{equation*}
and
\begin{equation*}
	\mathbb{V}\mathrm{ar}X^{\nu_1,\nu_2}(s,t)=\frac{\lambda s^{2\nu_1+1}t^{2\nu_2+1}}{\prod_{i=1}^{2}\Gamma(2\nu_i+1)\Gamma^2(\nu_i+1)},\ (s,t)\in\mathbb{R}^2_+,
\end{equation*}
respectively. For $\nu_1=\nu_2=1$, the integral (\ref{frint}) reduces to the Riemann integral of PRF defined by
\begin{equation}\label{rint}
	X(s,t)\coloneqq\int_{0}^{s}\int_{0}^{t}N(x,y)\,\mathrm{d}x\,\mathrm{d}y,\ (s,t)\in\mathbb{R}^2_+.
\end{equation}

The following result extends the result of \cite{Xia2018}, to the case of two-parameter L\'evy processes.
\begin{proposition}\label{lemma1}
	Let $\{X(s,t),\ (s,t)\in\mathbb{R}^2_+\}$ be as defined in (\ref{rint}). Then, its characteristic function has the following representation:
	\begin{equation}\label{lemma1res}
		\mathbb{E}\exp\big(i\xi X(s,t)\big)=\exp\bigg(\lambda st\int_{0}^{1}\int_{0}^{1}(e^{i\xi st xy}-1)\,\mathrm{d}x\,\mathrm{d}y\bigg),\ (s,t)\in\mathbb{R}^2_+,\ \xi\in\mathbb{R}.
	\end{equation}
\end{proposition}
\begin{proof}
	For any $(s,t)\in\mathbb{R}^2_+$, we have
	\begin{align*}
		X(s,t)&=\int_{0}^{t}\lim_{k\to\infty}\frac{s}{k}\sum_{r=1}^{k}N(rs/k,y)\,\mathrm{d}y\\
		&=\lim_{k\to\infty}\lim_{m\to\infty}\frac{st}{km}\sum_{r=1}^{k}\sum_{j=1}^{m}N(rs/k,jt/m)\\
		&=\lim_{k\to\infty}\lim_{m\to\infty}\frac{st}{km}\sum_{r=1}^{k}\sum_{j=1}^{m}(m-j+1)(N(rs/k,jt/m)-N(rs/k,(j-1)t/m))\\
		&=\lim_{k\to\infty}\lim_{m\to\infty}\frac{st}{km}\sum_{r=1}^{k}\sum_{j=1}^{m}(k-r+1)(m-j+1)(N(rs/k,jt/m)-N((r-1)s/k,jt/m)\\
		&\hspace{5cm}-N(rs/k,(j-1)t/m)+N((r-1)s/k,(j-1)t/m))\\
		&=\lim_{k\to\infty}\lim_{m\to\infty}\frac{st}{km}\sum_{r=1}^{k}\sum_{j=1}^{m}(k-r+1)(m-j+1)\Delta_{(r-1)s/k,(j-1)t/m}N(rs/k,jt/m).
	\end{align*}
	From  independent and stationary rectangular increments properties of PRF, the characteristic function of $X(s,t)$ is given by
	\begin{align*}
		\mathbb{E}e^{i\xi X(s,t)}&=\lim_{k\to\infty}\lim_{m\to\infty}\prod_{r=1}^{k}\prod_{j=1}^{m}\mathbb{E}\exp\bigg(i\frac{\xi st}{km}(k-r+1)(m-j+1)N(s/k,t/m)\bigg)\\
		&=\lim_{k\to\infty}\lim_{m\to\infty}\mathbb{E}\exp\bigg(i\frac{\xi st}{km}\sum_{r=1}^{k}\sum_{j=1}^{m}(k-r+1)(m-j+1)N(s/k,t/m)\bigg)\\
		&=\lim_{k\to\infty}\lim_{m\to\infty}\mathbb{E}\exp\bigg(i\frac{\xi st}{km}\sum_{r=1}^{k}\sum_{j=1}^{m}rjN(s/k,t/m)\bigg)\\
		&=\exp\bigg(\lim_{k\to\infty}\lim_{m\to\infty}\sum_{r=1}^{k}\sum_{j=1}^{m}\ln\mathbb{E}\exp\bigg(i\frac{\xi strj}{km}N(s/k,t/m)\bigg)\bigg)\\
		&=\exp\bigg(\lim_{k\to\infty}\lim_{m\to\infty}\frac{st}{km}\sum_{r=1}^{k}\sum_{j=1}^{m}\ln\mathbb{E}\exp\bigg(i\frac{\xi strj}{km}N(1,1)\bigg)\bigg)\\
		&=\exp\bigg(\lim_{k\to\infty}\lim_{m\to\infty}\frac{st}{km}\sum_{i=1}^{k}\sum_{j=1}^{m}\ln\exp\big(\lambda(e^{i\xi strj/km}-1)\big)\bigg),
	\end{align*}
	where the last two steps follow from (\ref{chprf}). This completes the proof.
\end{proof}
\begin{remark}\label{rem5.1}
	A similar identity as in (\ref{lemma1res}) can be obtained for the integral of a general two parameter L\'evy process (for existence and characterization see \cite{Straf1972}). Let $\{Y(s,t),\ (s,t)\in\mathbb{R}^2_+\}$ be a two parameter real valued L\'evy process with rectangular increment (for definition see (\ref{recinc})). Then, from the result of \cite{Straf1972}, it follows that $\mathbb{E}\exp(i\xi Y(s,t))=(\phi(\xi))^{st}$, $\xi\in\mathbb{R}$, where $\phi$ is the characteristics function of some infinitely divisible distribution on $\mathbb{R}$. Then, following the proof of Proposition  \ref{lemma1}, it can be established that 
	\begin{equation}\label{lemma1rem}
		\mathbb{E}\exp\bigg(i\xi\int_{0}^{t}\int_{0}^{s}Y(x,y)\,\mathrm{d}x\,\mathrm{d}y\bigg)=\exp\bigg(st\int_{0}^{1}\int_{0}^{1}\ln\phi(\xi st xy)\,\mathrm{d}x\,\mathrm{d}y\bigg),\ \xi\in\mathbb{R}.
	\end{equation}
	In particular, in the case of PRF, we have $\phi(\xi)=\exp(\lambda(e^{i\xi}-1))$, which on substituting in (\ref{lemma1rem}) yields (\ref{lemma1res}).
\end{remark}

Let $\{S(s,t),\ (s,t)\in\mathbb{R}^2_+\}$ be the two parameter generalized Skellam process as defined in (\ref{gsrf2}). For $\nu_i>0$, $i=1,2$, its fractional integral is defined as 
\begin{equation*}
	S^{\nu_1,\nu_2}(s,t)\coloneqq\frac{1}{\Gamma(\nu_1)\Gamma(\nu_2)}\int_{0}^{s}\int_{0}^{t}(s-x)^{\nu_1-1}(t-y)^{\nu_2-1}S(x,y)\,\mathrm{d}x\,\mathrm{d}y,\ (s,t)\in\mathbb{R}^2_+.
\end{equation*}
It is given by 
\begin{equation*}
		S^{\nu_1,\nu_2}(s,t)=\sum_{j\in\mathcal{J}}jX_j^{\nu_1,\nu_2}(s,t),\ (s,t)\in\mathbb{R}^2_+,
\end{equation*}
where $X_j^{\nu_1,\nu_2}(s,t)$ is the Riemann-Liouville integral of the PRF $N_j(s,t)$ as defined in (\ref{frint}) for each $j\in\mathcal{ J}$. Hence, its mean and variance are 
\begin{equation*}
	\mathbb{E}S^{\nu_1,\nu_2}(s,t)=\frac{\sum_{j\in\mathcal{J}}j\lambda_j s^{\nu_1+1}t^{\nu_2+1}}{\Gamma(\nu_1+2)\Gamma(\nu_2+2)}\ \text{and}\ \mathbb{V}\mathrm{ar}S^{\nu_1,\nu_2}(s,t)=\frac{\sum_{j\in\mathcal{J}}j^2\lambda_j s^{2\nu_1+1}t^{2\nu_2+1}}{\prod_{i=1}^{2}\Gamma(2\nu_i+1)\Gamma^2(\nu_i+1)},\ (s,t)\in\mathbb{R}^2_+,
\end{equation*}
respectively.
\begin{proposition}
	The characteristic function of Riemann integral of GSRF $\{\mathcal{S}(s,t),\ (s,t)\in\mathbb{R}^2_+\}$ defined in (\ref{gsrf2}) has the following representation:
	\begin{equation*}
		\mathbb{E}\exp\bigg(i\xi\int_{0}^{s}\int_{0}^{t}S(x,y)\,\mathrm{d}x\,\mathrm{d}y\bigg)=\exp\bigg(\sum_{j\in\mathcal{J}}\lambda_j st\int_{0}^{1}\int_{0}^{1}(e^{ij\xi st xy}-1)\,\mathrm{d}x\,\mathrm{d}y\bigg),\ (s,t)\in\mathbb{R}^2_+,\ \xi\in\mathbb{R}.
	\end{equation*}
\end{proposition}
\begin{proof}
	In view of Proposition \ref{lemma1}, its proof follows from the definition of $S(s,t)$.
\end{proof}

The following result show that the integral of SRF equals in distribution to a scaled compound Poisson random field. A similar result is holds for the case of one parameter generalized Skellam process (see \cite[Proposition 3.2]{Cinque2025}).
\begin{proposition}\label{prop5.2}
	Let $\{N(s,t),\ (s,t)\in\mathbb{R}^2_+\}$ be a PRF with parameter $\lambda>0$. Let us consider a compound Poisson random field $Y(s,t)=\sum_{r=1}^{N(s,t)}X_r$, where $\{X_r\}_{r\ge1}$ is sequence of iid random variables independent of the PRF. Then, its integral over rectangle $[0,s]\times[0,t]$ satisfies the following equality:
	\begin{equation*}
		\int_{0}^{s}\int_{0}^{t}Y(x,y)\,\mathrm{d}x\,\mathrm{d}y\overset{d}{=}st\sum_{r=1}^{N(s,t)}X_rU_r,\ (s,t)\in\mathbb{R}^2_+,
	\end{equation*}
	where $U_1,U_2,\dots$ are iid $Uniform\,[0,1]\times[0,1]$ random variables that are independent of the PRF  $\{N(s,t),\ (s,t)\in\mathbb{R}^2_+\}$ and $\{X_r\}_{r\ge1}$.
\end{proposition}
\begin{proof}
	In \cite{Vishwakarma2025b}, it is shown that the compound Poisson random field is a two parameter L\'evy process. Hence, from Remark \ref{rem5.1}, it follows that 
	\begin{align*}
		\mathbb{E}\exp\bigg(i\xi\int_{0}^{s}\int_{0}^{t}Y(x,y)\,\mathrm{d}x\,\mathrm{d}y\bigg)&=\exp\bigg(st\int_{0}^{1}\int_{0}^{1}\ln\mathbb{E}e^{i\xi st xyY(1,1)}\,\mathrm{d}x\,\mathrm{d}y\bigg),\ \xi\in\mathbb{R},\\
		&=\exp\bigg(st\int_{0}^{1}\int_{0}^{1}\ln\mathbb{E}\big(\mathbb{E}e^{i\xi st xyX_1}\big)^{N(1,1)}\,\mathrm{d}x\,\mathrm{d}y\bigg)\\
		&=\exp\bigg(st\lambda\int_{0}^{1}\int_{0}^{1}(\mathbb{E}e^{i\xi st xyX_1}-1)\,\mathrm{d}x\,\mathrm{d}y\bigg)\\
		&=\exp\bigg(st\lambda\bigg(\int_{0}^{1}\int_{0}^{1}\mathbb{E}e^{i\xi st xyX_1}\,\mathrm{d}x\,\mathrm{d}y-1\bigg)\bigg)\\
		&=\mathbb{E}\bigg(\int_{0}^{1}\int_{0}^{1}\mathbb{E}e^{i\xi st xyX_1}\,\mathrm{d}x\,\mathrm{d}y\bigg)^{N(s,t)}\\
		&=\mathbb{E}\big(\mathbb{E}e^{i\xi stX_1U_1}\big)^{N(s,t)}=\mathbb{E}\exp\bigg(i\xi st\sum_{r=1}^{N(s,t)}X_rU_r\bigg).
	\end{align*}
	This completes the proof.
\end{proof}
\begin{remark}
	Note that the  GSRF $\{\mathcal{S}(s,t),\ (s,t)\in\mathbb{R}^2_+\}$ as defined in (\ref{gsrf2}) satisfies $\mathcal{S}(s,t)\overset{d}{=}\sum_{r=1}^{N(s,t)}\mathcal{Y}_r$, where $\mathcal{Y}_r$'r are iid random variables as defined in Proposition \ref{gsrfcrep} and $\{N(s,t),\ (s,t)\in\mathbb{R}^2_+\}$ is a PRF with parameter $\sum_{j\in\mathcal{J}}\lambda_j>0$. Hence, from Remark \ref{rem5.1}, we have 
	\begin{equation*}
		\int_{0}^{s}\int_{0}^{t}\mathcal{S}(x,y)\,\mathrm{d}x\,\mathrm{d}y\overset{d}{=}\int_{0}^{s}\int_{0}^{t}\sum_{r=1}^{N(x,y)}\mathcal{Y}_r\,\mathrm{d}x\,\mathrm{d}y,\ (s,t)\in\mathbb{R}^2_+.
	\end{equation*}
	Thus, from Proposition \ref{prop5.2}, we get
	\begin{equation*}
		\int_{0}^{s}\int_{0}^{t}\mathcal{S}(x,y)\,\mathrm{d}x\,\mathrm{d}y\overset{d}{=}st\sum_{r=1}^{N(s,t)}\mathcal{Y}_rU_r,\ (s,t)\in\mathbb{R}^2_+.
	\end{equation*}
	where $U_r$'s are iid $Uniform\,[0,1]\times[0,1]$ random variables that are independent of  $\{N(s,t),\ (s,t)\in\mathbb{R}^2_+\}$ and $\{\mathcal{Y}_r\}_{r\ge1}$. 
\end{remark}
\section{Fractional Skellam random fields on $\mathbb{R}^2_+$}\label{sec4}
In this section, we consider three different fractional variants of the Skellam random field on the positive quadrant of the plane. In the first type, we define it as a time-changed SRF, where the time-changing component is a bivariate random process whose marginals are independent inverse stable subordinators.

\subsection{Fractional Skellam random field-I} Let $\{E_1^{\alpha}(t),\ t\ge0\}$ and $\{E_2^{\beta}(t),\ t\ge0\}$ be inverse stable subordinators of indices $\alpha\in(0,1)$ and $\beta\in(0,1)$, respectively. Let $\{S(s,t),\ (s,t)\in\mathbb{R}^2_+\}$ be a SRF on $\mathbb{R}^2_+$. We consider a random field $\{S^{\alpha,\beta}(s,t),\ (s,t)\in\mathbb{R}^2_+\}$ defined by
\begin{equation}\label{fsrf1}
	S^{\alpha,\beta}(s,t)\coloneqq S(E_1^{\alpha}(s),E_2^{\beta}(t))=N_1(E_1^{\alpha}(s),E_2^{\beta}(t))-N_2(E_1^{\alpha}(s),E_2^{\beta}(t)).
\end{equation}
It is assumed that all the component processes appearing in (\ref{fsrf1}) are independent of each other. We call it the fractional Skellam random field of type one (FSRF-I). Moreover, for $\alpha=\beta=1$, we have $S^{1,1}(s,t)=S(s,t)$.
\begin{remark}
	Let $\{Y_j\}_{j\ge1}$ be iid random variable as defined in Remark \ref{rem3.4}, and $\{N(s,t),\ (s,t)\in\mathbb{R}^2_+\}$ be the PRF. Then, the following equality holds:
$
		S^{\alpha,\beta}(s,t)\overset{d}{=}\sum_{j=1}^{N(E_1^\alpha(s),E_2^\beta(t))}Y_j,
$
	where the  $\{E_1^{\alpha}(t),\ t\ge0\}$ and $\{E_2^{\beta}(t),\ t\ge0\}$ are inverse stable subordinators. All the variables appearing here are mutually independent. Thus, the FSRF-I is a time fractional compound Poisson random field in the sense of \cite{Vishwakarma2025b}.
\end{remark}
The following result shows that the FSRF-I is a time-changed Skellam process where the time changing component is a two parameter random process defined as a product of two independent inverse stable subordinators.
\begin{proposition}
	Let $\{E_1^{\alpha}(t),\ t\ge0\}$ and $\{E_2^{\beta}(t),\ t\ge0\}$ be independent inverse stable subordinators which are independent of the Skellam process $\{S(t),\ t\ge0\}$ defined in (\ref{spdef}). Then, FSRF-I satisfies the following time-changed representation:
	\begin{equation*}
		S^{\alpha,\beta}(s,t)\overset{d}{=}S(E_1^{\alpha}(s)E_2^{\beta}(t)),\ (s,t)\in\mathbb{R}^2_+.
	\end{equation*}
\end{proposition}
\begin{proof}
	For $(s,t)\in\mathbb{R}^2_+$, on using (\ref{sppgf}), we have
	\begin{align*}
		\mathbb{E}\exp(&-uS(E_1^{\alpha}(s)E_2^{\beta}(t)))\\
		&=\int_{0}^{\infty}\int_{0}^{\infty}e^{-uS(xy)}\mathrm{Pr}\{E_1^\alpha(s)\in\mathrm{d}x\}\mathrm{Pr}\{E_2^\beta(t)\in\mathrm{d}y\}\\
		&=\int_{0}^{\infty}\int_{0}^{\infty}\exp\big(\lambda_1xy(e^{-u}-1)+\lambda_2xy\big(e^u-1\big)\big)\mathrm{Pr}\{E_1^\alpha(s)\in\mathrm{d}x\}\mathrm{Pr}\{E_2^\beta(t)\in\mathrm{d}y\}\\
		&=\int_{0}^{\infty}\int_{0}^{\infty}e^{-uS(x,y)}\mathrm{Pr}\{E_1^\alpha(s)\in\mathrm{d}x\}\mathrm{Pr}\{E_2^\beta(t)\in\mathrm{d}y\}\\
		&=\mathbb{E}\exp(-uS(E_1^\alpha(s),E_2^\beta(t))),
	\end{align*}
	where we have used (\ref{srfpgf}) to obtain the penultimate step.
	This completes the proof.
\end{proof}

The following result provides the governing equation for the pgf of FSRF-I.
\begin{proposition}
	The pgf $G^{\alpha,\beta}(u,s,t)=\mathbb{E}u^{S^{\alpha,\beta}(s,t)}$, $0<u\leq1$ of FSRF-I solves the following fractional differential equation:
	\begin{align*}
		\frac{\partial^{\alpha+\beta}}{\partial t^{\beta}\partial s^{\alpha}}G^{\alpha,\beta}(u,s,t)&=\bigg(\lambda_1(u-1)+\lambda_2\bigg(\frac{1}{u}-1\bigg)\bigg)G^{\alpha,\beta}(u,s,t)\\
		&\hspace{2.8cm} +\frac{(\lambda_1(u-1)u+\lambda_2(1-u))^2}{\lambda_1 u^2-\lambda_2}\frac{\partial}{\partial u}G^{\alpha,\beta}(u,s,t),\ \lambda_1 u^2-\lambda_2\ne0,
	\end{align*}
	with $G^{\alpha,\beta}(u,0,t)=G^{\alpha,\beta}(u,s,0)=1$.
\end{proposition}
\begin{proof}
	On using (\ref{fsrf1}), the pgf of FSRF-I is given by
	\begin{equation*}
		G^{\alpha,\beta}(u,s,t)=\int_{0}^{\infty}\int_{0}^{\infty}G(u,x,y)\mathrm{Pr}\{E_1^{\alpha}(s)\in\mathrm{d}x\}\mathrm{Pr}\{E_2^{\beta}(t)\in\mathrm{d}y\},
	\end{equation*}
	whose double Laplace transform is
	\begin{multline}\label{pf10}
		\int_{0}^{\infty}\int_{0}^{\infty}e^{-ws-zt}G^{\alpha,\beta}(u,s,t)\,\mathrm{d}s\,\mathrm{d}t\\=w^{\alpha-1}z^{\beta-1}\int_{0}^{\infty}\int_{0}^{\infty}G(u,x,y)e^{-w^{\alpha}x-z^{\beta}y}\,\mathrm{d}x\,\mathrm{d}y,\ w>0,\ z>0,
	\end{multline}
	where we have used (\ref{islap}). 
	
	Now, on taking the Laplace transform on both sides of (\ref{pgf2}) with respect to $t$, we get
	\begin{align}\label{pf11}
		z\frac{\partial}{\partial s}\int_{0}^{\infty}e^{-zt}G(u,s,t)\,\mathrm{d}t&=\bigg(\lambda_1(u-1)+\lambda_2\bigg(\frac{1}{u}-1\bigg)\bigg)\int_{0}^{\infty}e^{-zt}G(u,s,t)\,\mathrm{d}t\nonumber\\
		&\ \ +\frac{(\lambda_1(u-1)u+\lambda_2(1-u))^2}{\lambda_1 u^2-\lambda_2}\frac{\partial}{\partial u}\int_{0}^{\infty}e^{-zt}G(u,s,t)\,\mathrm{d}t,
	\end{align}
	where we have used $\frac{\partial}{\partial s}G(u,s,0)=0$. Here, interchange of derivative and integral is justified because the SRF have finite mean for all $(s,t)\in\mathbb{R}^2_+$. Again, taking the Laplace transform on both sides of (\ref{pf11}) with respect to $s$, we get
	\begin{align*}
		zw\int_{0}^{\infty}\int_{0}^{\infty}e^{-ws-zt}&G(u,s,t)\,\mathrm{d}s\,\mathrm{d}t-1\\
		&=\bigg(\lambda_1(u-1)+\lambda_2\bigg(\frac{1}{u}-1\bigg)\bigg)\int_{0}^{\infty}\int_{0}^{\infty}e^{-ws-zt}G(u,s,t)\,\mathrm{d}s\,\mathrm{d}t\nonumber\\
		&\ \ +\frac{(\lambda_1(u-1)u+\lambda_2(1-u))^2}{\lambda_1 u^2-\lambda_2}\frac{\partial}{\partial u}\int_{0}^{\infty}\int_{0}^{\infty}e^{-ws-zt}G(u,s,t)\,\mathrm{d}s\,\mathrm{d}t,
	\end{align*}
	where we have used that $\int_{0}^{\infty}e^{-zt}G(u,0,t)\,\mathrm{d}t=1/z$.
	So,
	\begin{align}
		z^{\beta}z^{\beta-1}w^{\alpha}&w^{\alpha-1}\int_{0}^{\infty}\int_{0}^{\infty}e^{-w^{\alpha}s-z^{\beta}t}G(u,s,t)\,\mathrm{d}s\,\mathrm{d}t-z^{\beta-1}w^{\alpha-1}\nonumber\\
		&=z^{\beta-1}w^{\beta-1}\bigg(\lambda_1(u-1)+\lambda_2\bigg(\frac{1}{u}-1\bigg)\bigg)\int_{0}^{\infty}\int_{0}^{\infty}e^{-w^{\beta}s-z^{\alpha}t}G(u,s,t)\,\mathrm{d}s\,\mathrm{d}t\nonumber\\
		&\ \ +z^{\beta-1}w^{\beta-1}\frac{(\lambda_1(u-1)u+\lambda_2(1-u))^2}{\lambda_1 u^2-\lambda_2}\frac{\partial}{\partial u}\int_{0}^{\infty}\int_{0}^{\infty}e^{-w^{\beta}s-z^{\alpha}t}G(u,s,t)\,\mathrm{d}s\,\mathrm{d}t.\label{pf12}
	\end{align}
	Now, on substituting (\ref{pf10}) in (\ref{pf12}), we have
	\begin{align}
		z^{\beta}w^{\alpha}\int_{0}^{\infty}&\int_{0}^{\infty}e^{-ws-zt}G^{\alpha,\beta}(u,s,t)\,\mathrm{d}s\,\mathrm{d}t-z^{\beta}w^{\alpha-1}\int_{0}^{\infty}e^{-zt}G^{\alpha,\beta}(u,0,t)\,\mathrm{d}t\nonumber\\
		&=\bigg(\lambda_1(u-1)+\lambda_2\bigg(\frac{1}{u}-1\bigg)\bigg)\int_{0}^{\infty}\int_{0}^{\infty}e^{-ws-zt}G^{\alpha,\beta}(u,s,t)\,\mathrm{d}s\,\mathrm{d}t\nonumber\\
		&\ \ +\frac{(\lambda_1(u-1)u+\lambda_2(1-u))^2}{\lambda_1 u^2-\lambda_2}\frac{\partial}{\partial u}\int_{0}^{\infty}\int_{0}^{\infty}e^{-ws-zt}G^{\alpha,\beta}(u,s,t)\,\mathrm{d}s\,\mathrm{d}t,\label{pf13}
	\end{align}
	where we have used that $\int_{0}^{\infty}e^{-zt}G^{\alpha,\beta}(u,0,t)\,\mathrm{d}t=1/z$.
	On taking the inverse Laplace transform on both sides of (\ref{pf13}) with respect to $s$ and using the following result (see \cite{Kilbas2006}):
	\begin{equation}\label{frderlap}
		\int_{0}^{\infty}e^{-wt}\frac{\mathrm{d}^\alpha}{\mathrm{d}t^\alpha}f(t)\,\mathrm{d}t=w^\alpha\int_{0}^{\infty}e^{-wt}f(t)\,\mathrm{d}t-w^{\alpha-1}f(0^+),\ w>0,\ \alpha>0,
	\end{equation}
	we get
	\begin{align}
		z^{\beta}\frac{\partial^{\alpha}}{\partial s^{\alpha}}\int_{0}^{\infty}e^{-zt}G^{\alpha,\beta}(u,s,t)\,\mathrm{d}t
		&=\bigg(\lambda_1(u-1)+\lambda_2\bigg(\frac{1}{u}-1\bigg)\bigg)\int_{0}^{\infty}e^{-zt}G^{\alpha,\beta}(u,s,t)\,\mathrm{d}t\nonumber\\
		&\ \ +\frac{(\lambda_1(u-1)u+\lambda_2(1-u))^2}{\lambda_1 u^2-\lambda_2}\frac{\partial}{\partial u}\int_{0}^{\infty}e^{-zt}G^{\alpha,\beta}(u,s,t)\,\mathrm{d}t,\label{pf14}
	\end{align}
	which again on taking inverse Laplace transform with respect to $z$, and using $\frac{\partial^{\alpha}}{\partial s^{\alpha}}G^{\alpha,\beta}(u,s,0)=0$ and (\ref{frderlap}) yields the required result. 
\end{proof}
\begin{theorem}
	The distribution $p^{\alpha,\beta}(n,s,t)=\mathrm{Pr}\{S^{\alpha,\beta}(s,t)=n\}$, $n\in\mathbb{Z}$ of FSRF-I reads
	\begin{align}
		p^{\alpha,\beta}(n,s,t)&=\bigg(\frac{\lambda_1}{\lambda_2}\bigg)^{n/2}\sum_{k=0}^{\infty}\frac{(\sqrt{\lambda_1\lambda_2}s^\alpha t^\beta)^{|n|+2k}}{(|n|+k)!k!}\nonumber\\
		&\hspace{1.5cm}\cdot{}_2\Psi_2\left[\begin{matrix}
			(|n|+2k+1,1)&(|n|+2k+1,1)\\\\
			((|n|+2k)\alpha+1,\alpha)&((|n|+2k)\beta+1,\beta)
		\end{matrix}\Bigg| -(\lambda_1+\lambda_2)s^{\alpha}t^{\beta} \right],\label{fsrf1dist}
	\end{align}
	where ${}_2\Psi_2$ is the generalized Wright function defined as follows (see \cite{Kilbas2006}):
	\begin{equation}\label{genwrit}
		{}_m\Psi_l\left[\begin{matrix}
			(a_1,\alpha_1),\,(a_2,\alpha_2),\dots,(a_m,\alpha_m)\\\\
			(b_1,\beta_1),\,(b_2,\beta_2),\dots,(b_l,\beta_l)
		\end{matrix}\Bigg| x \right]=\sum_{n=0}^{\infty}\frac{\prod_{i=1}^{m}\Gamma(a_i+n\alpha_i)x^n}{\prod_{j=1}^{l}\Gamma(b_j+n\beta_j)n!},\ x\in\mathbb{R},
	\end{equation}
	for $a_i,b_j\in\mathbb{R}$ and $\alpha_i,\beta_j\in\mathbb{R}-\{0\}$, $i=1,2,\dots,m$, $j=1,2,\dots,l$.
\end{theorem}
\begin{proof}
	From (\ref{fsrf1}), the distribution of FSRF-I is given by
	\begin{align}
		p^{\alpha,\beta}&(n,s,t)\nonumber\\
		&=\int_{0}^{\infty}\int_{0}^{\infty}p(n,x,y)\mathrm{Pr}\{E_1^{\alpha}(s)\in\mathrm{d}x\}\mathrm{Pr}\{E_2^{\beta}(t)\in\mathrm{d}y\}\label{depf1}\\
		&=\bigg(\frac{\lambda_1}{\lambda_2}\bigg)^{n/2}\sum_{k=0}^{\infty}\frac{(\sqrt{\lambda_1\lambda_2})^{|n|+2k}}{(|n|+k)!k!}\int_{0}^{\infty}\int_{0}^{\infty}e^{-(\lambda_1+\lambda_2)xy}(xy)^{|n|+2k}\mathrm{Pr}\{E_1^{\alpha}(s)\in\mathrm{d}x\}\mathrm{Pr}\{E_2^{\beta}(t)\in\mathrm{d}y\}.\nonumber
	\end{align}
	On using (\ref{islap}) and Fubini's theorem, its double Laplace transform is 
	\begin{align}
		\int_{0}^{\infty}\int_{0}^{\infty}&e^{-sw-zt}p^{\alpha,\beta}(n,s,t)\,\mathrm{d}s\,\mathrm{d}t\nonumber\\
		&=\bigg(\frac{\lambda_1}{\lambda_2}\bigg)^{n/2}\sum_{k=0}^{\infty}\frac{(\sqrt{\lambda_1\lambda_2})^{|n|+2k}}{(|n|+k)!k!}z^{\beta-1}w^{\alpha-1}\int_{0}^{\infty}\int_{0}^{\infty}e^{-(\lambda_1+\lambda_2)xy}(xy)^{|n|+2k}e^{-z^{\beta}y-w^{\alpha}x}\,\mathrm{d}x\,\mathrm{d}y\nonumber\\
		&=\bigg(\frac{\lambda_1}{\lambda_2}\bigg)^{n/2}\sum_{k=0}^{\infty}\frac{(\sqrt{\lambda_1\lambda_2})^{|n|+2k}}{(|n|+k)!k!}z^{\beta-1}\int_{0}^{\infty}y^{|n|+2k}e^{-z^{\beta}y}\frac{\Gamma(|n|+2k+1)w^{\alpha-1}}{(w^{\alpha}+(\lambda_1+\lambda_2)y)^{|n|+2k+1}}\,\mathrm{d}y.\label{pf21}
	\end{align}
	For $\alpha>0$, $\beta\in\mathbb{R}$ and $\gamma\in\mathbb{R}$, we have the following Laplace transform (see \cite{Kilbas2006}):
	\begin{equation}\label{tmllap}
		\int_{0}^{\infty}e^{-wt}t^{\beta-1}E_{\alpha,\beta}^\gamma(ct^\alpha)\,\mathrm{d}t=\frac{w^{\alpha\gamma-\beta}}{(w^\alpha-c)^\gamma},\ |cw^{-\alpha}|<1,\ w>0,
	\end{equation}
	where $E_{\alpha,\beta}^\gamma(\cdot)$ is the three parameter Mittag-Leffler function defined as follows:
	\begin{equation}\label{tml}
		E_{\alpha,\beta}^\gamma(x)=\sum_{r=0}^{\infty}\frac{(\gamma)_rx^k}{\Gamma(\alpha r+\beta)r!},\ x\in\mathbb{R}.
	\end{equation}
	Here, $(\gamma)_r=\gamma(\gamma+1)\dots(\gamma+r-1)$. 
	
	On taking the inverse Laplace transform on both sides of (\ref{pf21}) with respect to $w$ and using (\ref{tmllap}), we get
	{\scriptsize\begin{align*}
		&\int_{0}^{\infty}e^{-zt}p^{\alpha,\beta}(n,s,t)\,\mathrm{d}t\\
		&=\bigg(\frac{\lambda_1}{\lambda_2}\bigg)^{n/2}\sum_{k=0}^{\infty}\frac{(\sqrt{\lambda_1\lambda_2})^{|n|+2k}\Gamma(|n|+2k+1)}{(|n|+k)!k!}z^{\beta-1}s^{(|n|+2k)\alpha}\int_{0}^{\infty}y^{|n|+2k}e^{-z^{\beta}y}E_{\alpha,(|n|+2k)\alpha+1}^{|n|+2k+1}(-(\lambda_1+\lambda_2)ys^{\alpha})\,\mathrm{d}y\\
		&=\bigg(\frac{\lambda_1}{\lambda_2}\bigg)^{n/2}\sum_{k=0}^{\infty}\frac{(\sqrt{\lambda_1\lambda_2})^{|n|+2k}\Gamma(|n|+2k+1)}{(|n|+k)!k!}z^{\beta-1}s^{(|n|+2k)\alpha}\sum_{r=0}^{\infty}\frac{(|n|+2k+1)_r(-(\lambda_1+\lambda_2)s^{\alpha})^r}{\Gamma(r\alpha+(|n|+2k)\alpha+1)r!}\int_{0}^{\infty}y^{|n|+2k+r}e^{-z^{\beta}y}\,\mathrm{d}y\\
		&=\bigg(\frac{\lambda_1}{\lambda_2}\bigg)^{n/2}\sum_{k=0}^{\infty}\frac{(\sqrt{\lambda_1\lambda_2})^{|n|+2k}\Gamma(|n|+2k+1)}{(|n|+k)!k!}s^{(|n|+2k)\alpha}\sum_{r=0}^{\infty}\frac{(|n|+2k+1)_r(-(\lambda_1+\lambda_2)s^{\alpha})^r}{\Gamma(r\alpha+(|n|+2k)\alpha+1)r!}\frac{\Gamma(|n|+2k+r+1)}{z^{\beta(|n|+2k+r)+1}}.
	\end{align*}}
	Its inverse Laplace transform reads
	{\scriptsize\begin{align*}
		p^{\alpha,\beta}&(n,s,t)\\
		&=\bigg(\frac{\lambda_1}{\lambda_2}\bigg)^{n/2}\sum_{k=0}^{\infty}\frac{(\sqrt{\lambda_1\lambda_2})^{|n|+2k}}{(|n|+k)!k!}s^{(|n|+2k)\alpha}\sum_{r=0}^{\infty}\frac{\Gamma(|n|+2k+1+r)(-(\lambda_1+\lambda_2)s^{\alpha})^r}{\Gamma(r\alpha+(|n|+2k)\alpha+1)r!}\frac{\Gamma(|n|+2k+r+1)t^{\beta(|n|+2k+r)}}{\Gamma(\beta(|n|+2k+r)+1)}\\
		&=\bigg(\frac{\lambda_1}{\lambda_2}\bigg)^{n/2}\sum_{k=0}^{\infty}\frac{(\sqrt{\lambda_1\lambda_2})^{|n|+2k}}{(|n|+k)!k!}(s^\alpha t^\beta)^{|n|+2k}\sum_{r=0}^{\infty}\frac{\Gamma(|n|+2k+1+r)\Gamma(|n|+2k+1+r)(-(\lambda_1+\lambda_2)s^{\alpha}t^{\beta})^r}{\Gamma((|n|+2k)\alpha+1+r\alpha)\Gamma(\beta(|n|+2k)+1+r\beta)r!}.
	\end{align*}}
	This gives the required result.
\end{proof}
\begin{remark}
	If we allow $\alpha=\beta=1$ in (\ref{fsrf1dist}) then it reduces to 
	\begin{align*}
		p^{1,1}(n,s,t)&=\bigg(\frac{\lambda_1}{\lambda_2}\bigg)^{n/2}\sum_{k=0}^{\infty}\frac{(\sqrt{\lambda_1\lambda_2})^{|n|+2k}}{(|n|+k)!k!}(st)^{|n|+2k}e^{-(\lambda_1+\lambda_2)st}\\
		&=e^{-(\lambda_1+\lambda_2)st}\bigg(\frac{\lambda_1}{\lambda_2}\bigg)^{n/2}I_{|n|}(2\sqrt{\lambda_1\lambda_2}st),\ n\in\mathbb{Z},
	\end{align*}
	which coincides with (\ref{srfdist}).
\end{remark}
	From (\ref{fsrf1}), the mean of FSRF-I is given by
	\begin{equation*}
		\mathbb{E}S^{\alpha,\beta}(s,t)=\int_{0}^{\infty}\int_{0}^{\infty}\mathbb{E}S(x,y)\mathrm{Pr}\{E_1^{\alpha}(s)\in\mathrm{d}x,E_2^{\beta}(t)\in\mathrm{d}y\},
	\end{equation*}
	whose double Laplace transform is
	\begin{align*}
		\int_{0}^{\infty}\int_{0}^{\infty}e^{-sw-zt}\mathbb{E}S^{\alpha,\beta}(s,t)\,\mathrm{d}s\,\mathrm{d}t&=(\lambda_1-\lambda_2)w^{\alpha-1}z^{\beta-1}\int_{0}^{\infty}\int_{0}^{\infty}xye^{-w^{\alpha}x}e^{-z^{\beta}y}\,\mathrm{d}x\,\mathrm{d}y\\
		&=\frac{(\lambda_1-\lambda_2)}{w^{\alpha+1}z^{\beta+1}}.
	\end{align*}
	Its double inversion yields 
	\begin{equation}\label{meanfsrf1}
		\mathbb{E}S^{\alpha,\beta}(s,t)=\frac{(\lambda_1-\lambda_2)s^{\alpha}t^{\beta}}{\Gamma(\alpha+1)\Gamma(\beta+1)},\ (s,t)\in\mathbb{R}^2_+.
	\end{equation}
	
	Similarly, the double Laplace transform of the second moment of FSRF-I is given by
	\begin{align*}
		\int_{0}^{\infty}\int_{0}^{\infty}e^{-sw-zt}&\mathbb{E}(S^{\alpha,\beta}(s,t))^2\,\mathrm{d}s\,\mathrm{d}t\\
		&=w^{\alpha-1}z^{\beta-1}\int_{0}^{\infty}\int_{0}^{\infty}((\lambda_1+\lambda_2)xy+(\lambda_1-\lambda_2)^2x^2y^2)e^{-w^\alpha x}e^{-z^\beta y}\,\mathrm{d}x\,\mathrm{d}y\\
		&=\frac{(\lambda_1+\lambda_2)}{w^{\alpha+1}z^{\beta+1}}+\frac{4(\lambda_1-\lambda_2)^2}{w^{2\alpha+1}z^{2\beta+1}},
	\end{align*}
	whose inverse Laplace transform yields
	\begin{equation*}
		\mathbb{E}(S^{\alpha,\beta}(s,t))^2=\frac{(\lambda_1+\lambda_2)s^\alpha t^\beta}{\Gamma(\alpha+1)\Gamma(\beta+1)}+\frac{4(\lambda_1-\lambda_2)^2s^{2\alpha}t^{2\beta}}{\Gamma(2\alpha+1)\Gamma(2\beta+1)}.
	\end{equation*}
	Thus, its variance is given by
	\begin{equation}\label{vfsrf1}
		\mathbb{V}\mathrm{ar}S^{\alpha,\beta}(s,t)=\frac{(\lambda_1+\lambda_2)s^\alpha t^\beta}{\Gamma(\alpha+1)\Gamma(\beta+1)}+\frac{4(\lambda_1-\lambda_2)^2s^{2\alpha}t^{2\beta}}{\Gamma(2\alpha+1)\Gamma(2\beta+1)}-\frac{(\lambda_1-\lambda_2)^2s^{2\alpha}t^{2\beta}}{\Gamma^2(\alpha+1)\Gamma^2(\beta+1)},\ (s,t)\in\mathbb{R}^2_+.
	\end{equation}
\begin{remark}
	The characteristic function of SRF is given by 
	\begin{equation*}
		\mathbb{E}e^{i\xi S(s,t)}=\exp\big(\lambda_1st(e^{i\xi}-1)+\lambda_2st(e^{-i\xi}-1)\big),\ \xi\in\mathbb{R}.
	\end{equation*}
	By taking $\psi(\xi)=\lambda_1(1-e^{i\xi})+\lambda_2(1-e^{-i\xi})$, we have $\mathbb{E}e^{i\xi S(s,t)}=e^{-st\psi(\xi)}$, $\xi\in\mathbb{R}$. As the SRF has independent and stationary rectangular increments, we can use the result of \cite{Vishwakarma2025} to derive the mean, variance and auto covariance of the FSRF-I. From Theorem 2.2 of \cite{Vishwakarma2025}, we have 
	\begin{equation*}
		\mathbb{E}S^{\alpha,\beta}(s,t)=\mathbb{E}S(1,1)\mathbb{E}E_1^{\alpha}(s)\mathbb{E}E_2^\beta(t)=\frac{(\lambda_1-\lambda_2)s^\alpha t^\beta}{\Gamma(\alpha+1)\Gamma(\beta+1)},\ (s,t)\in\mathbb{R}^2_+,
	\end{equation*}
	where we have used (\ref{ismeam}). Also, we have
	\begin{equation*}
		\mathbb{V}\mathrm{ar}S^{\alpha,\beta}(s,t)=\mathbb{V}\mathrm{ar}S(1,1)\mathbb{E}E_1^\alpha(s)\mathbb{E}E_2^\beta+(\mathbb{E}S(1,1))^2\mathbb{V}\mathrm{ar}E_1^{\alpha}(s)E_2^\beta(t).
	\end{equation*}
	From Example 2.3 of \cite{Vishwakarma2025}, we get
	\begin{equation*}
		\mathbb{V}\mathrm{ar}E_1^{\alpha}(s)E_2^\beta(t)=s^{2\alpha}t^{2\beta}\bigg(\frac{4}{\Gamma(2\alpha+1)\Gamma(2\beta+1)}-\frac{1}{\Gamma^2(\alpha+1)\Gamma^2(\beta+1)}\bigg).
	\end{equation*}
	Thus, 
	\begin{equation*}
		\mathbb{V}\mathrm{ar}S^{\alpha,\beta}(s,t)=\frac{(\lambda_1+\lambda_2)s^\alpha t^\beta}{\Gamma(\alpha+1)\Gamma(\beta+1)}+(\lambda_1-\lambda_2)^2s^{2\alpha}t^{2\beta}\bigg(\frac{4}{\Gamma(2\alpha+1)\Gamma(2\beta+1)}-\frac{1}{\Gamma^2(\alpha+1)\Gamma^2(\beta+1)}\bigg).
	\end{equation*}
	These results coincides with results obtained in (\ref{meanfsrf1}) and (\ref{vfsrf1}).
	
	Further, for $(s,t)\preceq(s',t')$, the auto covariance of FSRF-I is given as follows:
	\begin{equation*}
		\mathbb{C}\mathrm{ov}(S^{\alpha,\beta}(s,t),S^{\alpha,\beta}(s',t'))=\mathbb{V}\mathrm{ar}S(1,1)\mathbb{E}E_1^\alpha(s)\mathbb{E}E_2^\beta(t)+(\mathbb{E}S(1,1))^2\mathbb{C}\mathrm{ov}(E_1^\alpha(s)E_2^\beta(t),E_1^\alpha(s')E_2^\beta(t')).
	\end{equation*}
	From Example 2.3 of \cite{Vishwakarma2025}, we have
	\begin{align*}
		\mathbb{C}\mathrm{ov}(E_1^\alpha(s)E_2^\beta(t),E_1^\alpha(s')E_2^\beta(t'))&=\frac{1}{\Gamma(\alpha+1)\Gamma(\alpha)}\int_{0}^{s}((s'-x)^\alpha+(s-x)^\alpha)x^{\alpha-1}\,\mathrm{d}x\\
		&\ \ \cdot \frac{1}{\Gamma(\beta+1)\Gamma(\beta)}\int_{0}^{t}((t'-y)^{\beta}+(t-y)^{\beta})y^{\beta-1}\,\mathrm{d}y\\
		&\ \ -\frac{(ss')^\alpha(tt')^\beta}{\Gamma^2(\alpha+1)\Gamma^2(\beta+1)},\ (s,t)\preceq(s',t').
	\end{align*}
	Thus,
	\begin{align*}
		\mathbb{C}\mathrm{ov}(S^{\alpha,\beta}(s,t),S^{\alpha,\beta}(s',t'))&=(\lambda_1-\lambda_2)^2\bigg(\frac{1}{\Gamma(\alpha+1)\Gamma(\alpha)}\int_{0}^{s}((s'-x)^\alpha+(s-x)^\alpha)x^{\alpha-1}\,\mathrm{d}x\\
		&\ \ \cdot \frac{1}{\Gamma(\beta+1)\Gamma(\beta)}\int_{0}^{t}((t'-y)^{\beta}+(t-y)^{\beta})y^{\beta-1}\,\mathrm{d}y\\
		&\ \ -\frac{(ss')^\alpha(tt')^\beta}{\Gamma^2(\alpha+1)\Gamma^2(\beta+1)}\bigg)+\frac{(\lambda_1+\lambda_2)s^\alpha t^\beta}{\Gamma(\alpha+1)\Gamma(\beta+1)},\ (s,t)\preceq(s',t').
	\end{align*}
\end{remark}

\subsection{Fractional Skellam random field-II} 
We now analyze the second variant of the fractional SRF, in which the time fractionality is introduced only in one component of the indexing parameter. Let $\{E^\alpha(s),\ s\ge0\}$ be an inverse $\alpha$-stable subordinator that is independent of the SRF $\{S(s,t),\ (s,t)\in\mathbb{R}^2_+\}$. We consider a random field $\{S^\alpha(s,t),\ (s,t)\in\mathbb{R}^2_+\}$ defined as follows:
\begin{equation}\label{fsrf2}
	S^\alpha(s,t)\coloneqq S(E^\alpha(s),t).
\end{equation}
We call it the fractional Skellam random field of type two (FSRF-II). Similar to (\ref{fsrf2}), we can also replace the second time component of SRF with an independent inverse stable subordinator. 
\begin{theorem}
	The distribution $p^\alpha(n,s,t)=\mathrm{Pr}\{S^\alpha(s,t)=n\}$, $n\in\mathbb{Z}$ of FSRF-II is given by
	\begin{equation*}
	p^{\alpha}(n,s,t)=\bigg(\frac{\lambda_1}{\lambda_2}\bigg)^{n/2}\sum_{k=0}^{\infty}\frac{(|n|+2k)!(\sqrt{\lambda_1\lambda_2}s^\alpha t)^{|n|+2k}}{(|n|+k)!k!}E_{\alpha,\alpha(|n|+2k)}^{|n|+2k+1}(-(\lambda_1+\lambda_2)s^\alpha t),\ n\in\mathbb{Z}.
	\end{equation*}
	
	Moreover, it solves the following system of fractional differential equations:
	\begin{equation*}
		\frac{\partial^\alpha}{\partial s^\alpha}p^{\alpha}(n,s,t)=-(\lambda_1+\lambda_2)tp^{\alpha}(n,s,t)+\lambda_1 tp^{\alpha}(n-1,s,t)+\lambda_2 tp^{\alpha}(n+1,s,t),\ n\in\mathbb{Z},
	\end{equation*}
	with initial condition $p^{\alpha}(0,0,t)=1$ for all $t\ge0$,
\end{theorem}
\begin{proof}
	From (\ref{fsrf2}), the distribution of FSRF-II is given by
	\begin{align}
		p^{\alpha}(n,s,t)&=\int_{0}^{\infty}p(n,x,t)\mathrm{Pr}\{E^{\alpha}(s)\in\mathrm{d}x\}\label{de2pf1}\\
		&=\bigg(\frac{\lambda_1}{\lambda_2}\bigg)^{n/2}\sum_{k=0}^{\infty}\frac{(\sqrt{\lambda_1\lambda_2}t)^{|n|+2k}}{(|n|+k)!k!}\int_{0}^{\infty}e^{-(\lambda_1+\lambda_2)xt}x^{|n|+2k}\mathrm{Pr}\{E^{\alpha}(s)\in\mathrm{d}x\},\nonumber
	\end{align}
	where we have used (\ref{srfdist}).
	Its Laplace transform with respect $s$ is
	\begin{align*}
		\int_{0}^{\infty}e^{-ws}p^{\alpha}(n,s,t)\,\mathrm{d}s
		&=\bigg(\frac{\lambda_1}{\lambda_2}\bigg)^{n/2}\sum_{k=0}^{\infty}\frac{(\sqrt{\lambda_1\lambda_2}t)^{|n|+2k}}{(|n|+k)!k!}w^{\alpha-1}\int_{0}^{\infty}e^{-(\lambda_1+\lambda_2)xt}x^{|n|+2k}e^{-xw^\alpha}\,\mathrm{d}x\\
		&=\bigg(\frac{\lambda_1}{\lambda_2}\bigg)^{n/2}\sum_{k=0}^{\infty}\frac{(\sqrt{\lambda_1\lambda_2}t)^{|n|+2k}}{(|n|+k)!k!}\frac{(|n|+2k)!w^{\alpha-1}}{(w^\alpha+(\lambda_1+\lambda_2)t)^{|n|+2k+1}}.
	\end{align*}
	By using (\ref{tmllap}), its inversion yields the required distribution of FSRF-II.
	
	The Laplace transform of (\ref{de2pf1}) is given by
	\begin{equation}\label{de2pf2}
		\int_{0}^{\infty}e^{-ws}p^{\alpha}(n,s,t)\,\mathrm{d}s=w^{\alpha-1}\int_{0}^{\infty}p(n,x,t)e^{-xw^\alpha}\,\mathrm{d}x,\ w>0.
	\end{equation}
	Now, on taking the Laplace transform on both sides of (\ref{difeq1}) with respect to $s$, we have
	\begin{align*}
		w\int_{0}^{\infty}e^{-ws}&p(n,s,t)\,\mathrm{d}s-p(n,0,t)\\
		&=\int_{0}^{\infty}e^{-ws}(-(\lambda_1+\lambda_2)tp(n,s,t)+\lambda_1 tp(n-1,s,t)+\lambda_2 tp(n+1,s,t))\,\mathrm{d}s,\ w>0.
	\end{align*}
	So,
	\begin{align}
		w^\alpha w^{\alpha-1}&\int_{0}^{\infty}e^{-w^\alpha s}p(n,s,t)\,\mathrm{d}s-w^{\alpha-1}p(n,0,t)\nonumber\\
		&=w^{\alpha-1}\int_{0}^{\infty}e^{-w^\alpha s}(-(\lambda_1+\lambda_2)tp(n,s,t)+\lambda_1 tp(n-1,s,t)+\lambda_2 tp(n+1,s,t))\,\mathrm{d}s.\label{de2pf3}
	\end{align}
	By using (\ref{de2pf2}) in (\ref{de2pf3}), and $p^\alpha(n,0,t)=p(n,0,t)$, we get
	\begin{align}
		w^\alpha \int_{0}^{\infty}&e^{-w s}p^\alpha(n,s,t)\,\mathrm{d}s-w^{\alpha-1}p^\alpha(n,0,t)\nonumber\\
		&=\int_{0}^{\infty}e^{-w s}(-(\lambda_1+\lambda_2)tp^\alpha(n,s,t)+\lambda_1 tp^\alpha(n-1,s,t)+\lambda_2 tp^\alpha(n+1,s,t))\,\mathrm{d}s.\label{de2pf4}
	\end{align}
	On using (\ref{frderlap}), the inversion of (\ref{de2pf4}) yields the required governing equations. This completes the proof.
\end{proof}
\begin{remark}
	The mean and variance of FSRF-II are given by
	\begin{equation*}
		\mathbb{E}S^\alpha(s,t)=\frac{(\lambda_1-\lambda_2)s^\alpha t}{\Gamma(\alpha+1)}
	\end{equation*}
	and
	\begin{equation*}
		\mathbb{V}\mathrm{ar}S^\alpha(s,t)=\frac{(\lambda_1+\lambda_2)s^\alpha t}{\Gamma(\alpha+1)}+(\lambda_1-\lambda_2)^2s^{2\alpha}t^{2}\bigg(\frac{2}{\Gamma(2\alpha+1)}-\frac{1}{\Gamma^2(\alpha+1)}\bigg),\ (s,t)\in\mathbb{R}^2_+,
	\end{equation*}
	respectively.
	These can be derived by using the first and second moments of SRF and definition (\ref{fsrf2}).
\end{remark}
\begin{remark}
	From (\ref{fsrf2}), the pgf of FSRF-II is given by
	\begin{equation*}
		\mathbb{E}u^{S^\alpha(s,t)}=\int_{0}^{\infty}G(u,x,t)\mathrm{Pr}\{E^\alpha(s)\in\mathrm{d}x\},
	\end{equation*}
	where $G(u,s,t)$ is the pgf of SRF. Its Laplace transform with respect to variable $s$ is given by
	\begin{align*}
		\int_{0}^{\infty}e^{-ws}\mathbb{E}u^{S^\alpha(s,t)}\,\mathrm{d}s&=w^{\alpha-1}\int_{0}^{\infty}\exp\bigg(\lambda_1xt(u-1)+\lambda_2xt\bigg(\frac{1}{u}-1\bigg)\bigg)e^{-xw^\alpha}\,\mathrm{d}x\\
		&=\frac{w^{\alpha-1}}{w^\alpha+\bigg(\lambda_1t(u-1)+\lambda_2t\bigg(\frac{1}{u}-1\bigg)\bigg)},\ w>0,
	\end{align*}
	whose inversion yields
	\begin{equation*}
		\mathbb{E}u^{S^\alpha(s,t)}=E_{\alpha,1}\bigg(\lambda_1s^\alpha t(u-1)+\lambda_2s^\alpha t\bigg(\frac{1}{u}-1\bigg)\bigg),\ 0<u\leq1.
	\end{equation*}
\end{remark}

\subsection{Fractional Skellam random field-III} We define the third fractional variant of the SRF as a difference of two independent fractional Poisson random fields on $\mathbb{R}^2_+$ (for definition, see Section \ref{fprf}). Let $\{N_1^{\alpha,\beta}(s,t),\ (s,t)\in\mathbb{R}^2_+\}$, $0<\alpha,\beta\leq1$ and  $\{N_2^{\alpha',\beta'}(s,t),\ (s,t)\in\mathbb{R}^2_+\}$, $0<\alpha',\beta'\leq1$ be two independent fractional Poisson random fields with parameter $\lambda_1>0$ and $\lambda_2>0$, respectively. Let us consider the random field $\{S_{\alpha',\beta'}^{\alpha,\beta}(s,t),\ (s,t)\in\mathbb{R}^2_+\}$ defined as follows:
 \begin{equation}
 	S_{\alpha',\beta'}^{\alpha,\beta}(s,t)\coloneqq N_1^{\alpha,\beta}(s,t)-N_2^{\alpha',\beta'}(s,t).
 \end{equation}
 We call it the fractional Skellam random field of third type (FSRF-III). For $n\ge0$, by using (\ref{fprfdist}), its distribution is given by
 {\footnotesize\begin{align*}
 	\mathrm{Pr}&\{S_{\alpha',\beta'}^{\alpha,\beta}(s,t)=n\}\\
 	&=\sum_{k=0}^{\infty}\mathrm{Pr}\{N_1^{\alpha,\beta}(s,t)=n+k\}\mathrm{Pr}\{N_2^{\alpha',\beta'}(s,t)=k\}\\
 	&=\sum_{k=0}^{\infty}\sum_{r=n+k}^{\infty}\frac{(-1)^{r-n-k}r_{(r-n-k)}r_{(n+k)}(\lambda_1 s^{\alpha}t^\beta)^r}{\Gamma(r\alpha+1)\Gamma(r\beta+1)}\sum_{l=k}^{\infty}\frac{(-1)^{l-k}l_{(l-k)}l_{(k)}(\lambda_2 s^{\alpha'}t^{\beta'})^l}{\Gamma(l\alpha'+1)\Gamma(l\beta'+1)}\\
 	&=\sum_{k=0}^{\infty}\sum_{r=0}^{\infty}\frac{(-1)^{r}(r+n+k)_{(r)}(r+n+k)_{(n+k)}(\lambda_1 s^{\alpha}t^\beta)^{r+n+k}}{\Gamma((r+n+k)\alpha+1)\Gamma((r+n+k)\beta+1)}\sum_{l=0}^{\infty}\frac{(-1)^{l}(l+k)_{(l)}(l+k)_{(k)}(\lambda_2 s^{\alpha'}t^{\beta'})^{l+k}}{\Gamma((l+k)\alpha'+1)\Gamma((l+k)\beta'+1)}\\
 	&=\sum_{r=0}^{\infty}\sum_{l=0}^{\infty}\frac{(-1)^{l+r}(\lambda_1s^\alpha t^\beta)^{r+n}(\lambda_2s^{\alpha'}t^{\beta'})^l}{r!l!}\sum_{k=0}^{\infty}\frac{\Gamma(r+n+1+k)\Gamma(r+n+1+k)\Gamma(l+1+k)}{\Gamma(n+1+k)\Gamma(1+k)\Gamma((r+n)\alpha+1+k\alpha)}\\
 	&\hspace{6.5cm}\cdot\frac{\Gamma(l+1+k)(\lambda_1\lambda_2s^{\alpha+\alpha'} t^{\beta+\beta'})^k}{\Gamma((r+n)\beta+1+k\beta)\Gamma(l\alpha'+1+k\alpha')\Gamma(l\beta'+1+k\beta')}\\
 	&=\sum_{r=0}^{\infty}\sum_{l=0}^{\infty}\frac{(-1)^{l+r}(\lambda_1s^\alpha t^\beta)^{r+n}(\lambda_2s^{\alpha'}t^{\beta'})^l}{r!l!}\\
 	&\ \ \cdot{}_4\Psi_5\left[\begin{matrix}
 		(r+n+1,1)&(r+n+1,1)&(l+1,1)&(l+1,1)\\\\
 		(n+1,1)&((r+n)\alpha+1,\alpha)&((r+n)\beta+1,\beta)&(l\alpha'+1,\alpha')&(l\beta'+1,\beta')
 	\end{matrix}\bigg|\lambda_1\lambda_2s^{\alpha+\alpha'}t^{\beta+\beta'}\right],
 \end{align*}}
where ${}_4\Psi_5$ is the generalized Wright function defined in (\ref{genwrit}). Similarly, for $n<0$, we have
{\footnotesize\begin{align*}
	\mathrm{Pr}&\{S_{\alpha',\beta'}^{\alpha,\beta}(s,t)=n\}\\
	&=\sum_{k=|n|}^{\infty}\mathrm{Pr}\{N_1^{\alpha,\beta}(s,t)=n+k\}\mathrm{Pr}\{N_2^{\alpha',\beta'}(s,t)=k\}\\
	&=\sum_{k=0}^{\infty}\mathrm{Pr}\{N_1^{\alpha,\beta}(s,t)=k\}\mathrm{Pr}\{N_2^{\alpha',\beta'}(s,t)=k+|n|\}\\
	&=\sum_{r=0}^{\infty}\sum_{l=0}^{\infty}\frac{(-1)^{l+r}(\lambda_2s^{\alpha'} t^{\beta'})^{r+|n|}(\lambda_1s^{\alpha}t^{\beta})^l}{r!l!}\\
	&\ \ \cdot{}_4\Psi_5\left[\begin{matrix}
		(r+|n|+1,1)&(r+|n|+1,1)&(l+1,1)&(l+1,1)\\\\
		(|n|+1,1)&((r+|n|)\alpha'+1,\alpha')&((r+|n|)\beta'+1,\beta')&(l\alpha+1,\alpha)&(l\beta+1,\beta)
	\end{matrix}\bigg|\lambda_1\lambda_2s^{\alpha+\alpha'}t^{\beta+\beta'}\right],
\end{align*}}
which can be obtained following the steps of case $n\ge0$.
\begin{remark}
	For $\lambda_1=\lambda_2=\lambda$, $\alpha=\alpha'$ and $\beta=\beta'$, we have symmetry in the expression of distribution of FSRF-III. Thus, it this case, it is given by
	{\footnotesize\begin{align*}
	\mathrm{Pr}&\{S_{\alpha',\beta'}^{\alpha,\beta}(s,t)=n\}\\
	&=\sum_{r=0}^{\infty}\sum_{l=0}^{\infty}\frac{(-1)^{l+r}(\lambda s^{\alpha} t^{\beta})^{r+|n|+l}}{r!l!}\\
	&\ \ \cdot{}_4\Psi_5\left[\begin{matrix}
		(r+|n|+1,1)&(r+|n|+1,1)&(l+1,1)&(l+1,1)\\\\
		(|n|+1,1)&((r+|n|)\alpha+1,\alpha)&((r+|n|)\beta+1,\beta)&(l\alpha+1,\alpha)&(l\beta+1,\beta)
	\end{matrix}\bigg|(\lambda s^{\alpha}t^{\beta})^2\right],\ n\in\mathbb{Z}.
	\end{align*}}
\end{remark}
\begin{remark}
	For $\alpha=\alpha'=\beta=\beta'=1$, the FSRF-III reduces to the SRF and we have
	\begin{align*}
		\mathrm{Pr}\{S_{1,1}^{1,1}(s,t)=n\}&=\sum_{r=0}^{\infty}\sum_{l=0}^{\infty}\frac{(-1)^{l+r}(\lambda_1s t)^{r+n}(\lambda_2st)^l}{r!l!}\sum_{k=0}^{\infty}\frac{(\lambda_1\lambda_2s^2t^2)^k}{(n+k)!k!}\\
		&=e^{-(\lambda_1+\lambda_2)st}\bigg(\frac{\lambda_1}{\lambda_2}\bigg)^{n/2}I_n(2\sqrt{\lambda_1\lambda_2}st),\ n\ge0.
	\end{align*}
	Also, for $n<0$, we have
	\begin{align*}
		\mathrm{Pr}\{S_{1,1}^{1,1}(s,t)=n\}&=\sum_{r=0}^{\infty}\sum_{l=0}^{\infty}\frac{(-1)^{l+r}(\lambda_2st)^{r+|n|}(\lambda_1st)^l}{r!l!}\sum_{k=0}^{\infty}\frac{(\lambda_1\lambda_2s^2t^2)^k}{(|n|+k)!k!}\\
		&=e^{-(\lambda_1+\lambda_2)st}\bigg(\frac{\lambda_2}{\lambda_1}\bigg)^{n/2}I_{|n|}(2\sqrt{\lambda_1\lambda_2}st),
	\end{align*}
	which coincides with (\ref{srfdist}).
\end{remark}

For $(s,t)$ and $(s',t')$ in $\mathbb{R}^2_+$, by using (\ref{fprfmean}), (\ref{fprfvar}) and (\ref{fprfcov}), the mean, variance and auto covariance  of FSRF-III are given by
\begin{equation*}
	\mathbb{E}S_{\alpha',\beta'}^{\alpha,\beta}(s,t)=\frac{\lambda_1s^\alpha t^\beta}{\Gamma(\alpha+1)\Gamma(\beta+1)}-\frac{\lambda_2s^{\alpha'}t^{\beta'}}{\Gamma(\alpha'+1)\Gamma(\beta'+1)},
\end{equation*}
\begin{align*}
	\mathbb{V}\mathrm{ar}S_{\alpha',\beta'}^{\alpha,\beta}(s,t)&=\frac{\lambda_1 s^\alpha t^\beta}{\Gamma(\alpha+1)\Gamma(\beta+1)}+\frac{(2\lambda_1 s^\alpha t^\beta)^2}{\Gamma(2\alpha+1)\Gamma(2\beta+1)}-\frac{(\lambda_1 s^\alpha t^\beta)^2}{\Gamma^2(\alpha+1)\Gamma^2(\beta+1)}\\
	&\ \ +\frac{\lambda_2 s^{\alpha'} t^{\beta'}}{\Gamma(\alpha'+1)\Gamma(\beta'+1)}+\frac{(2\lambda_2 s^{\alpha'} t^{\beta'})^2}{\Gamma(2\alpha'+1)\Gamma(2\beta'+1)}-\frac{(\lambda_2 s^{\alpha'} t^{\beta'})^2}{\Gamma^2(\alpha'+1)\Gamma^2(\beta'+1)}
\end{align*}
and
\begin{align*}
	\mathbb{C}\mathrm{ov}(S_{\alpha',\beta'}^{\alpha,\beta}(s,t),S_{\alpha',\beta'}^{\alpha,\beta}(s',t'))&=\frac{\lambda_1(s\wedge s')^\alpha(t\wedge t')^\beta}{\Gamma(\alpha+1)\Gamma(\beta+1)}-\frac{\lambda_1^2(ss')^\alpha(tt')^\beta}{\Gamma^2(\alpha+1)\Gamma^2(\beta+1)}\nonumber\\
	&\ \ +\frac{\lambda_1}{\alpha\Gamma^2(\alpha)}\int_{0}^{s\wedge s'}((s-x)^\alpha+(s'-x)^\alpha)x^{\alpha-1}\,\mathrm{d}x\nonumber\\
	&\ \ \cdot\frac{\lambda_1}{\beta\Gamma^2(\beta)}\int_{0}^{t\wedge t'}((t-y)^\beta+(t'-y)^\beta)y^{\beta-1}\,\mathrm{d}y\\
	&\ \ +\frac{\lambda_2(s\wedge s')^{\alpha'}(t\wedge t')^{\beta'}}{\Gamma(\alpha'+1)\Gamma(\beta'+1)}-\frac{\lambda_2^2(ss')^{\alpha'}(tt')^{\beta'}}{\Gamma^2(\alpha'+1)\Gamma^2(\beta'+1)}\nonumber\\
	&\ \ +\frac{\lambda_2}{\alpha'\Gamma^2(\alpha')}\int_{0}^{s\wedge s'}((s-x)^{\alpha'}+(s'-x)^{\alpha'})x^{\alpha'-1}\,\mathrm{d}x\nonumber\\
	&\ \ \cdot\frac{\lambda_2}{\beta'\Gamma^2(\beta')}\int_{0}^{t\wedge t'}((t-y)^{\beta'}+(t'-y)^{\beta'})y^{\beta'-1}\,\mathrm{d}y,
\end{align*}
respectively.

\vspace*{5mm}
\paragraph{\textbf{Acknowledgment}} This work is partially supported by the National Post Doctoral Fellowship, PDF/2025/000076, from Anusandhan National Research Foundation, Govt. of India.

\end{document}